\documentclass{jotart}

\usepackage[english]{babel}
\usepackage{amsmath, amsthm, amssymb, amsfonts}
\usepackage{latexsym}
\usepackage{graphicx}
\usepackage{pdfpages}
\usepackage {bm}
\usepackage {indentfirst} 

\usepackage{hyperref,cite}
\usepackage{mathrsfs}

\newcommand{\De}{\mathbb{D}}
\newcommand{\D}{\mathcal{D}}
\newcommand{\C}{\mathbb{C}}
\newcommand{\N}{\mathbb{N}}

\newcommand{\B}{\mathcal{B(H)}}
\newcommand{\R}{\mathcal{R}}
\newcommand{\n}{\mathcal{N}}
\newcommand{\h}{\mathcal{H}}
\newcommand{\m}{\mathcal{M}}
\newcommand{\ka}{\mathcal{K}}

\theoremstyle{proclaim}
\newtheorem{theorem}{Theorem}[section]

\newtheorem{corollary}[theorem]{Corollary}
\newtheorem{proposition}[theorem]{Proposition}
\theoremstyle{statement}
\newtheorem{remark}[theorem]{Remark}
\newtheorem{definition}[theorem]{Definition}
\newtheorem{example}[theorem]{Example}

\theoremstyle{fancyproclaim}

\numberwithin{equation}{section}

\begin{document}
\issueinfo{00}{0}{0000} 
\commby{Hari Bercovici}
\pagespan{101}{130}
\date{Month dd, yyyy}
\revision{Month dd, yyyy}
\title[Operators with two-isometric dilations]{Hilbert space operators with two-isometric dilations}
\author[C. Badea{\protect \and}L. Suciu]{C\u{a}t\u{a}lin Badea{\protect
\and}Laurian Suciu}
\address{C. BADEA, Univ Lille, CNRS, UMR 8524 - Laboratoire Paul Painlev\'{e}, France}
\email{cbadea@univ-lille.fr}
\address{L. SUCIU, Dept. of Mathematics and Informatics, ``Lucian Blaga'' University of Sibiu, Romania}
\email{laurians2002@yahoo.com}

\begin{abstract} 
A continuous linear Hilbert space operator $S$ is said to be a $2$-isometry if the operator $S$ and its adjoint $S^*$ satisfy the relation $S^{*2}S^{2} - 2 S^{*}S + I = 0$. In this paper, we study
Hilbert space operators having liftings or dilations to $2$-isometries. The adjoint of an operator which admits such liftings is characterized as the restriction of a backward shift on a Hilbert space of vector-valued analytic functions. These results are applied to concave operators (i.e., operators $S$ such that $S^{*2}S^{2} - 2 S^{*}S + I \le 0$) and to operators similar to contractions or isometries. Two types of liftings to $2$-isometries, as well as the extensions induced by them, are constructed and isomorphic minimal liftings are discussed. 
\end{abstract}
\maketitle
\section{INTRODUCTION AND PRELIMINARIES}\label{sect:intro}

\subsection{Overview.} Beginning with the Sz.-Nagy dilation theorem, isometric liftings and unitary dilations of Hilbert space contractions (operators of norm no greater than one) have been basic objects of study in Operator Theory. This is witnessed, for instance, by the central role played by isometric liftings and unitary dilations in the celebrated monographs \cite{FF} and \cite{SFb}. For further considerations it is useful to keep in mind that every isometry can be written as a direct sum of a shift and a unitary operator (the Wold decomposition theorem) and that contractions have isometric liftings and unitary dilations.

Starting with a series of three papers \cite{AS1,AS2,AS3} by Agler and Stankus, a rich theory of $2$-isometries has been developed in recent years. A $2$-{\it isometry} is a Hilbert space operator $T$ which satisfies the second order difference condition $\|T^2x\|^2 - 2\|Tx\|^2 + \|x\|^2 = 0$ for every $x$, instead of the classical, first order condition $\|Tx\|^2 - \|x\|^2=0$ satisfied by isometries. The Dirichlet shift is an example of a $2$-isometry which is not an isometry. Oper\-a\-tors arising from a certain class of nonstationary stochastic processes related to Brownian motion (Brownian unitaries) play an essential role in the theory of $2$-isometries of Agler and Stankus, the same that unitary operators play for isometries. The fact that a $2$-isometry has an extension to a Brownian unitary has been proved in \cite[Theorem 5.80]{AS2}. As an analogue of the Wold decomposition theorem for isometries, it has been proved in \cite{Ol} (see also \cite{Richter0})  that a pure $2$-isometry is unitarily equivalent to a shift operator (multiplication by the independent
variable) on a Dirichlet space $D(\mu)$ corresponding to a positive operator measure $\mu$ on the unit circle.

The aim of this paper is to undertake a systematic study of operators possessing $2$-isometric liftings. This class of operators, denoted $\mathscr{C}_2$, can be viewed as the class of ``$2$-contractions''. To give a flavor of the results obtained in this paper we mention now several sample results. In Theorem~\ref{te48}, we give a characterization of adjoints of operators in $\mathscr{C}_2$ as restrictions of a backward shift operator on some Hilbert spaces of vector-valued analytic functions. We also prove that operators in $\mathscr{C}_2$ are compressions to semi-invariant subspaces of \emph{analytic} Brownian unitaries. An analogue of the von Neumann inequality is obtained in Theorem \ref{te48}, (iii), with the shift operator on a Dirichlet space $D(\mu)$ as an extremal operator. For a bounded linear {\it concave operator}, $T$, acting on a complex Hilbert space $\h$, that is one which satisfies the condition
$$\|T^2x\|^2 - 2\|Tx\|^2 + \|x\|^2 \le 0\quad  \textrm{ for every } x\in \h,$$
it is proved in Theorem~\ref{pr210} below that $T$ has a $2$-isometric lifting $S$, acting on a larger Hilbert space $\ka$, such that $S|_{\ka \ominus \h}$
is an isometry on $\ka \ominus  \h$ and $\|S^*S-I\| = \|T^*T-I\|$. Moreover, the $2$-isometric lifting $S$ is minimal in the usual sense.

\subsection{Basic definitions and notation.} Throughout this paper $\mathcal{B}(\h,\h')$ denotes the Banach space of all bounded linear operators acting from a complex Hilbert space $\h$ into another one, $\h'$, and $\B$ is a short for $\mathcal{B}(\h,\h)$. For an operator $T\in \mathcal{B}(\h,\h')$ its \emph{adjoint operator} in $\mathcal{B}(\h',\h)$ is denoted by $T^*$, while $\R(T)$ and $\n(T)$ stand for the \emph{range}, respectively the \emph{kernel} of $T$. An operator $T$ is a \emph{contraction} if $T^*T \le I$, where $I=I_{\h}$ is the identity operator on $\h$. The operator $T$ is an \emph{isometry} if $T^*T=I$ and $T$ is \emph{unitary} if it is an isometry with $\R(T)=\h'$. For a closed subspace $\m$ of $\h$, the \emph{orthogonal projection} in $\B$ with the range $\m$ is denoted by $P_{\m}$. The subspace $\m$ is \emph{invariant} (\emph{reducing}) for $T\in \B$ if $T\m \subset \m$ (respectively, $T\m\subset \m$ and $T^*\m \subset \m$). An isometry $T$ on $\h$ is pure (or a \emph{shift} operator) if there is no subspace $\m \neq \{0\}$ in $\h$ that reduces $T$ to a unitary operator.

An operator $T$ on $\h$ is said to be \emph{positive} (in notation $T\ge 0$) if $\langle Th,h\rangle \ge 0$ for every $h\in \h$. Here $\langle \cdot, \cdot \rangle$ denotes the inner product. When $T$ is positive, $T^{1/2}$ stands for the \emph{positive square root} of $T$. For a contraction $T\in \mathcal{B}(\h,\h')$, the operator $D_T=(I_{\h}-T^*T)^{1/2}\in \B$ and its closed range $\D_T=\overline{\R(D_T)}$ are called the \emph{defect operator}, respectively the \emph{defect space}, of $T$.

An operator $Z \in \mathcal{B}(\h,\h')$ is \emph{invertible} whenever $\n(Z)=\{0\}$ and $\R(Z)=\h'$. An operator $Z \in \mathcal{B}(\h,\h')$ \emph{intertwines} $T\in \B$ with $T' \in \mathcal{B}(\h')$ if $ZT=T'Z$. The operators $T$ and $T'$ are \emph{similar}, respectively \emph{unitarily equivalent}, whenever an intertwining relation $ZT=T'Z$ holds with an invertible, respectively unitary operator $Z$.

Whenever $ZT=T'Z$ with $Z$ an isometry, one says that $T'$ is an {\it extension} of $T$ (to $\h'$), or that $T$ is a {\it restriction} of $T'$ (on $\h$). If this holds one has $T^*Z^*=Z^*T'^*$ and one says that $T'^*$ is a {\it lifting} of $T^*$. In these cases $T$ is unitarily equivalent to $T_0':=T'|_{\R(Z)}$, while $T^*$ is unitarily equivalent to $T'^*_0$. Thus, if we identify $\h$ with $Z\h$ into $\h'$, then $T'$ becomes an extension of $T$ from the subspace $\h$ of $\h'$ to $\h'$, while $Z$ is the embedding mapping $J_{\h,\h'}$ of $\h$ into $\h'$. In this way we can consider only extensions and liftings of operators on $\h$ to another Hilbert spaces containing $\h$. Then the intertwining relations become $J_{\h,\h'}T=T'J_{\h,\h'}$, respectively $P_{\h',\h}T'^*=T^*P_{\h',\h}$, where $P_{\h',\h}=J_{\h,\h'}^* \in \mathcal{B}(\h',\h)$ is the projection of $\h'$ onto $\h$, while the inclusion $\h \subset \h'$ means that $\h$ is a closed subspace of $\h'$.

According to  \cite{FF}, two liftings $S_1$ on $\ka_1 \supset \h$ and $S_2$ on $\ka_2\supset \h$ of an operator $T$ on $\h$ are called {\it isomorphic} if there exists a unitary operator $Z\in \mathcal{B}(\ka_1,\ka_2)$ such that $ZS_1=S_2Z$ and $Z|_{\h}=I$, i.e. $S_1$ and $S_2$ are unitarily equivalent by $Z$ and $Z$ fixes the elements of $\h$. 

For a lifting $S$ on $\ka$ of $T$ one considers the subspace $\ka_0=\bigvee_{n\ge 0} S^n\h$, i.e. the smallest invariant subspace for $S$ in $\ka$ that contains $\h$. Let $S_0=S|_{\ka_0}$ be the restriction of $S$ to $\ka_0$. Clearly $S_0$ is also a lifting of $T$ on $\ka_0$, and it is called a {\it minimal lifting} of $T$. 

	Let $\h$ be a closed subspace of a Hilbert space $\ka$. The \emph{compression} of $R\in \mathcal{B}(\ka)$ to $\h$ is defined as $T=P_{\ka,\h}RJ_{\h,\ka}( = P_{\h}R|_{ \h})$. Recall the following useful result due to Sarason (see \cite{SFb}). The operator $R$ is a \emph{(power) dilation} of $T$, that is, $T^n = P_{\h}R^n\mid \h$ for all positive integers $n$, if and only if the subspace $\h$ is \emph{semi-invariant} for $R$, that is $\h = \h_1 \ominus \h_2$ for two invariant subspaces $\h_1$ and $\h_2$ of $R$.

According to \cite{Sh}, if $T\in \B$ is a fixed \emph{left invertible} operator (i.e. injective and with closed range), then $T'=T(T^*T)^{-1}$ is called the {\it Cauchy dual operator} of $T$. Clearly, $T'$ is also left invertible and $T^*T'=T'{^*}T=I$. So $T$ is the Cauchy dual operator of $T'$ and $\n(T^*)=\n(T'^{*})=:\mathcal{E}$. Obviously, $\mathcal{E}$ is a \emph{wandering subspace} for $T$, that is $\mathcal{E} \perp T^n\mathcal{E}$ for each $n\ge 1$.

\subsection{Two-isometries} In this paper we are interested in $2$-isometric liftings for operators in $\B$.
We refer to \cite{AS1, AS2, AS3, Aleman, ACJS, AR, BS2, GO, McC, Ol, Richter, Richter0, Rydhe, MMS, MajSu, Sh, Stankus, S-2020} for different aspects of $2$-isometries.

Recall that an operator $T$ on $\h$ is a 2-{\it isometry} if it satisfies the condition $T^*\Delta_TT=\Delta_T$, where $\Delta_T=T^*T-I$. In this case $\Delta_T$ is called the {\it operator of covariance} and ${\rm cov}(T)=\|\Delta_T\|^{1/2}$ is the {\it covariance} of $T$. Clearly, ${\rm cov}(T)=0$ if and only if $T$ is an isometry. If $T$ is a 2-isometry, then $\Delta_T\ge 0$, which means that $T$ is an {\it expansive} operator. A $2$-isometry $T$ that is \emph{power bounded}, i.e. $\sup_{n\in \N}\|T^n\|<\infty$, is necesarily an isometry. Also, $\n(\Delta_T)$ is an invariant subspace for a 2-isometry $T$ and $V=T|_{\n(\Delta_T)}$ is an isometry, so the canonical matrix representation of $T$ on $\h=\n(\Delta_T) \oplus \overline{\R(\Delta_T)}$ has the form
\begin{equation}\label{eq12}
T=
\begin{pmatrix}
V & E\\
0 & Y
\end{pmatrix},
\end{equation}
with $V^*E=0$. This yields $\Delta_T=0\oplus (E^*E+\Delta_Y)$, with $\Delta_0=E^*E+\Delta_Y=\Delta_T|_{\overline{\R(\Delta_T)}}$ an injective operator. Moreover, using that $T^*\Delta_TT=\Delta_T$, one obtains $Y^*\Delta_0Y=\Delta_0$.

Some special 2-isometries introduced by Agler and Stankus, called Brownian isometries and unitaries, play a special role in our investigations. Namely, $T$ is called a {\it Brownian isometry} if in the block matrix \eqref{eq12} one has $E=\delta E_0$, where $E_0$ is an injective contraction from $\overline{\R(\Delta_T)}$ into $\n(V^*)$, $\delta={\rm cov}(T)$, and $Y$ is a unitary operator which commutes with $E^*E$ (cf. \cite[Proposition 5.37]{AS2}). Also, $T$ is called a {\it Brownian unitary} if $E_0$ is an isometry with $\R(E_0)=\n(V^*)$ and $Y$ is unitary (cf. \cite[Proposition 5.12]{AS2}).

Using the above relationships between $T$ and $\Delta_T$, respectively $Y$ and $\Delta_0$, we see that 2-isometries are closely related to $A$-contractions, defined as follows.
When $T,A \in \B$ with $A\ge 0$, the operator $T$ is said to be an $A$-{\it contraction} if $T^*AT\le A$ and $T$ is an $A$-{\it isometry} if $T^*AT=A$. Thus, $2$-isometries are $\Delta_T$-isometries and the concave operators are $\Delta_T$-contractions, but many other classes of operators can be viewed as $A$-contractions for suitable operators $A$. $A$-contractions and $A$-isometries frequently appear in operatorial interpolation and robust control problems, based on the commutant lifting theory (see \cite{BFF,FF,FFK2}), or in other topics of operator theory \cite{BS1, BS2,BFF,CS,Do,JKKL,Ke,Kub-1997,MMS,MajSu,McC,McCR,MS,BP,Rydhe,Stankus,SS,S-2020,SuSu,S-2009,TV}.

\subsection{Organization of the paper} Following \cite{Sh}, we give in the next section a characterization of the adjoints of operators in the class $\mathscr{C}_2$, as restrictions of a backward shift on some spaces of vector-valued analytic functions. This implies that operators in $\mathscr{C}_2$ always have analytic $2$-isometric liftings. An analogue of the von Neumann inequality for operators in the class $\mathscr{C}_2$ is also obtained. Furthermore, in the case of analytic Brownian unitaries, we refine the Richter-Olofsson model  by taking into account additional  spectral information.

In Section \ref{sect:type0} we describe those operators having $2$-isometric liftings among the operators in the class of $A$-isometries $T$ with $A\ge \Delta_T$. Special features of such liftings, called of type I, are investigated. We apply our results to a class of expansive operators containing concave operators, and to operators similar to isometries. In particular, for these expansive operators we give a generalization of the extension theorem of Agler and Stankus \cite[Theorem 5.80]{AS2}. The use of the Treil-Volberg generalization of the commutant lifting theorem is to be mentioned here.

In Section \ref{sect:type1} we study operators $T\in \mathscr{C}_2$ in the context of $A$-contractions with $A\ge \Delta_T$. Such liftings, called of type II, are more general than those from Section \ref{sect:type0}. We show that these liftings, like those from Section \ref{sect:type0}, can be always chosen to be minimal. The results can be applied to operators similar to contractions. We also discuss some conditions for two minimal 2-isometric liftings to be isomorphic. Two examples are given in order to show that some operators similar to contractions may or may not have 2-isometric liftings of type I.

\subsection{An open problem} We end this Introduction by mentioning the following problem. 
Let $M>0$ be a real number. Consider the class $\mathscr{C}_2(M)$ of Hilbert space operators which have $2$-isometric liftings of covariance less than or equal to $M$. Then $\mathscr{C}_2(M)$ is a \emph{family} in the sense of Agler's abstract approach to model theory (see \cite{Ag}). It is an interesting problem to find the boundary and the extremal elements of the family $\mathscr{C}_2(M)$ (see \cite{Ag} for the undefined terms). We hope to return to this problem in the future.

\section{CHARACTERIZATIONS OF $\mathscr{C}_2$ AND ANALYTIC $2$-ISOMETRIC LIFTINGS}\label{sect:analytic}
Assume that the left invertible operator $S\in \B$ is analytic, that is
$$\h_{\infty} = \bigcap_{n\ge 1} S^n\h =\{0\}.$$
Let $S'=S(S^*S)^{-1}$ be the Cauchy dual operator of $S$. Then, as in \cite{Sh}, one can associate to $\h$ a Hilbert space of $\mathcal{E}$-valued analytic functions $\Theta h$, for $h\in \h$, where $\mathcal{E}=\n(S^*)$ and
$$
(\Theta h)(z)=\sum_{n\ge0} (P_{\mathcal{E}}S'^{*n}h)z^n, \qquad |z|<r(S')^{-1},
$$
$r(S')$ being the spectral radius of $S'$. Consider $\mathcal{D}:=\{\Theta h: h\in \h\}$ as a Hilbert space with the norm induced by $\h$. In this case, the operator $\Theta$ from $\h$ onto $\D$ is an isometry such that $\Theta S=M\Theta$, where $M$ is the forward shift on $\D$, i.e. $(Md)(z)=zd(z)$ for $d\in \D$. Also, one has $\Theta S'^{*}=B\Theta$, where $B$ is the backward shift on $\D$, i.e. $(Bd)(z)=\frac{1}{z}(d(z)-d(0))$, $d\in \D$.

It was proved in \cite[Theorem 3.6]{Sh} that a concave operator $S$, not necessary analytic, admits a Wold-type decomposition, in the sense that
$$
\h=\h_{\infty} \oplus \bigvee_{n\ge0} S^n\mathcal{E}.
$$
In this case the operators $S$ and $S'$ are simultaneously analytic. When this happens, one can associate to $\h$ another Hilbert space $\D'$ of $\mathcal{E}$-valued analytic functions $\Theta 'h$, where
$$
(\Theta 'h)(z)=\sum_{n\ge0} (P_{\mathcal{E}}S^{*n}h)z^n, \qquad h\in \h, \quad |z|<1.
$$
The operator $\Theta '$ isometrically maps $\h$ onto $\D'$ such that $\Theta 'S'=M'\Theta '$ and $\Theta 'S^*=B'\Theta '$, where $M'(B')$ is the forward (backward) shift on $\D'$. The space $\D '$ is the dual of $\D$ with respect to the Cauchy pairing (see also \cite{Sh}). In this way $S^*$ can be identified (by $\Theta$) with the adjoint of the forward shift $M$ on $\D$, or with the backward shift $B'$ on $\D'$ (by $\Theta '$). It is also known that the space $\D'$ is a Bergman space if and only if $\Delta_S=I-S'S'^*$ (see \cite{Ol} and \cite[Theorem 3.1]{GO}).

In addition, if $S$ is a 2-isometry, then, by \cite{Ol}, the associated space $\D$ becomes a Dirichlet type space $\D(\mu)$ of $\mathcal{E}$-valued analytic functions on the open unit disk $\De$ obtained with respect to some positive $\mathcal{B}(\mathcal{E})$-valued operator measure $\mu$ on the unit circle $\mathbb{T}$. We refer to \cite[(3.1)]{Ol} for a description of the norm in $\D(\mu)$ induced by $\mu$.

In general, if $M$ is the forward shift on a Dirichlet type space $\D(\mu)$, then $M$ is an analytic 2-isometry (see \cite[Theorem 3.1]{Ol}), while the operator $M'$ can be seen as the forward shift on the associated Hilbert space $\D'$ of $\n(M^*)$-valued analytic functions on $\De$, $D'$ being the dual of $\D(\mu)$ with respect to the Cauchy pairing. As above, the backward shift $B'$ in $\D'$ is unitarily equivalent to $M^*$ in $\D(\mu)$.

The following result characterizes membership into the class $\mathscr{C}_2$.

\begin{theorem}\label{te48}
For $T\in \B$ the following statements are equivalent:
\begin{itemize}
\item[(i)] $T$ is an element of the class $\mathscr{C}_2$, that is $T$ has a 2-isometric lifting;
\item[(ii)] $T$ has an analytic 2-isometric lifting;
\item[(iii)] there is an operator-valued positive measure $\mu$ on $\mathbb{T}$ such that the inequality
$$\left\|\left[ p_{ij}(T)\right]\right\|_{\mathcal{M}_n(\B)} \le \left\|\left[ p_{ij}(M_z)\right]\right\|_{\mathcal{M}_n(\mathcal{B}(\D(\mu)))} $$
holds true for all finite matrices of polynomials $\left[ p_{ij}\right]$, where $M_z$ is the multiplication by the variable $z$ on $\D(\mu)$ and $\mathcal{M}_n(\B)$, the set of $n\times n$ matrices with entries in $\B$,  is identified with the set of bounded linear operators acting on $\h^{(n)}$, the $\ell_2$-sum of $n$ copies of $\h$.
\item[(iv)] $T$ has a Brownian unitary (power) dilation;
\item[(v)] $T$ has an analytic Brownian unitary dilation;
\item[(vi)] $T^*$ is unitarily equivalent to the restriction to an invariant subspace of the backward shift $B'$ on a Hilbert space $\D'$ of vector-valued analytic functions on $\De$, $\D'$ being in Cauchy pairing to a Dirichlet type space $\D(\mu)$.
\end{itemize}

Moreover, if these conditions hold, then the lifting in (ii) for $T$ can be chosen minimal. Also, the liftings and Brownian unitary dilations in (i), (ii), (iv) and (v) can be chosen to have the same covariance.
\end{theorem}

\begin{proof}
The implication (i) $\Rightarrow$ (iv) follows from \cite[Theorem 5.80]{AS2}, using the fact that an extension of a lifting for $T$ gives a (power) dilation for $T$.

Let us now assume (iv), that is $T$ has a Brownian unitary dilation $S$ on $\ka \supset \h$. Because $S$ is not necessarily analytic we can write $S$ as a direct sum $S=U_0 \oplus S_1$ on $\ka=\ka_{\infty} \oplus  \ka_1$, where $\ka_{\infty}=\bigcap_{n\ge 1} S^n\ka$ and $U_0=S|_{\ka_{\infty}}$ is unitary, while $S_1=S|_{\ka_1}$ is an analytic Brownian unitary. But $U_0$ can be lifted to a Brownian unitary $S_0$ on a Hilbert space 
$$\ka_0:=\ell_+^2(\R(E))\oplus \ka_{\infty},$$ where $E\in \mathcal{B}(\ka_{\infty})$ is an isometry. Then $S_0$ has the following block matrix form
$$
S_0=
\begin{pmatrix}
S_+ & \widetilde{E}\\
0 & U_0
\end{pmatrix},
$$
with $S_+$ the forward shift on $\ka'_0=\ell_+^2 (\R(E))$ and $\widetilde{E}=\delta JE$, where $\delta ={\rm cov}(S)$ and $J$ is the embedding of $\R(E)$ into $\ka_0'$. Clearly, $S_0$ has not any isometric summand ($S_+$ being a shift with $R(E)=\n(S_+^*)$), hence $S_0$ is analytic.

Consider the operator $\widetilde{S}=S_0\oplus S_1$ on $\widetilde{\ka}:=\ka_0 \oplus \ka_1$. To see that $\widetilde{S}$ is a Brownian unitary operator we use the matrix representation of $S_1$ on $\ka_1=\n(\Delta_{S_1}) \oplus \R(\Delta_{S_1})=: \ka_2\oplus \ka_3$ (similar to that of $S_0$ from above) given by a forward shift $S_+'$ on $\ka_2$, a unitary operator $U_1$ on $\ka_3$ and an operator $\widetilde{E}_1=\delta E_1$, with $\delta$ as above and $E_1$ an isometry from $\ka_3$ onto $\n(S_+'^{*})$. Thus we can represent $\widetilde{S}$ on $\widetilde{\ka}=\ka_0' \oplus \ka_2 \oplus \ka_{\infty} \oplus \ka_3=(\ka_0' \oplus \ka_2) \oplus (\ka_{\infty} \oplus \ka_3)$, respectively, in the form
\begin{equation}\label{eq:page7}
\widetilde{S}=
\begin{pmatrix}
S_+ & 0 & \widetilde{E} & 0\\
0 & S_+' & 0 & \widetilde{E}_1\\
0 & 0 & U_0 & 0\\
0 & 0 & 0 & U_1
\end{pmatrix}
=
\begin{pmatrix}
S' & \widetilde{G}\\
0 & U
\end{pmatrix}.
\end{equation}
Here $S'= S_+ \oplus S_+'$ is a forward shift, $U=U_0 \oplus U_1$ is a unitary and $\widetilde{G}=\delta G$ with $G$ an isometry from $\ka_{\infty} \oplus \ka_3$ onto $\n(S'^*)$ in $\ka_0'\oplus \ka_2$, $\delta$ being as above. Hence $\widetilde{S}$ is an analytic Brownian unitary with ${\rm cov}(\widetilde{S})=\delta= {\rm cov}(S)$. Since $S$ is a power dilation for $T$ we infer from the representation (\ref{eq:page7}) that $\widetilde{S}$ is also a power dilation for $T$. This concludes the proof that (iv) implies (v).

It is a known fact that from a power dilation $\widetilde{S}$ for $T$ one can obtain a lifting $S$ for $T$, as a restriction of $\widetilde{S}$ to an invariant subspace (see \cite{FF}). Furthermore, if $\widetilde{S}$ is Brownian unitary, then $S$ is a 2-isometry, and $S$ is analytic if $\widetilde{S}$ is, because the analyticity is preserved on invariant subspaces. Hence (v) implies (ii).

Next, we assume (ii). Let $S$ be an analytic 2-isometric lifting on $\ka \supset \h$ for $T$. Then, by \cite[Theorem 4.1]{Ol}, and using the same notation as in the Introduction, we obtain a $\mathcal{B}(\mathcal{E})$-valued positive measure $\mu$ on the unit circle such that $S$ is unitarily equivalent to the forward shift $M = M_z$ on a Dirichlet space $D(\mu)$ of $\mathcal{E}$-valued analytic functions on $\De$, where $\mathcal{E}=\n(S^*)$. It follows that $M$ is an extremal operator in the sense of (iii). Conversely, if (iii) holds, then, using the terminology of \cite{B1}, $T$ is completely polynomially dominated by $M$. Using \cite[Theorem 2.1.3]{B1}, we get that $T$ is unitarily equivalent to the compression of an operator $R$ to a semi-invariant subspace, $R$ being the image of $M$ by a unital C$^*$-representation. It follows that $R$ is a $2$-isometry, and thus $T$ is in the class $\mathscr{C}_2$, that is $T$ satisfies (i). So the equivalences (i)-(v) are proved.

Next, we assume (ii). Then, again, $S$ is an analytic 2-isometric lifting on $\ka \supset \h$ for $T$, unitarily equivalent to the forward shift $M$ on a Dirichlet space $D(\mu)$ of $\mathcal{E}$-valued analytic functions on $\De$. Also, $S^*$ is unitarily equivalent to the backward shift $B'$ on the Hilbert space $\D'=\{\Theta 'k :k\in \ka\}$ of $\mathcal{E}$-valued analytic functions of the form
$$
(\Theta 'k)(z)=\sum_{n\ge 1} (P_{\mathcal{E}}S^{*n}k)z^n, \qquad |z|<1.
$$
Since $S^*|_{\h}=T^*$, it follows that $B'\Theta 'h=\Theta 'T^*h$ for $h \in \h$. Hence, the closed subspace $\D'_0=\{\Theta 'h:h\in \h\}$ is invariant for $B'$. Because the operator $\Theta '$ from $\ka$ onto $\D'$ is an isometry and $\Theta 'S^*=B'\Theta '$, we infer that $\Theta '_0=\Theta '|_{\h}$ is an isometry from $\h$ onto $\D '_0$ and $\Theta '_0T^*=(B'|_{\D'_0})\Theta '_0$. Hence $T^*$ is unitarily equivalent to $B'|_{\D'_0}$, and we conclude that (ii) implies (vi).

Conversely, let us assume that $T^*$ is unitarily equivalent to $B'_0=B'|_{\D'_0}$, where $B'$ is as in (vi) and $\D'_0$ is an invariant subspace of $B'$. The Cauchy pairing between the space $\D'$, where $B'$ acts, and the Dirichlet type space $\D(\mu)$ ensures that $B'$ is unitarily equivalent to $M^*$, the adjoint of the forward shift on $\D(\mu)$. Since $M$ is a 2-isometry on $\D(\mu)$ by \cite[Theorem 3.1]{Ol}, it induces a 2-isometry $B'^{*}$ on $\D'$ which is a lifting for $B'^*_0$ ($B'$ being an extension of $B'_0$). In the same time, $T$ is unitarily equivalent to $B'^*_0$ by our assumption. We deduce that $T$ has $B'^{*}$ as a analytic 2-isometric lifting on $\D'$. We have thus shown that (vi) implies (ii). In conclusion, all equivalences (i)-(vi) are now proved.

Remark that, by \cite[Theorem 5.80]{AS2}, a 2-isometry $S$ has a Brownian unitary extension $\widetilde{S}$ which preserves the covariance of $S$, and, as we already have seen, ${\rm cov}(S)$ can be also preserved for an analytic Brownian unitary extension of $S$. On the other hand, if $S$ is a 2-isometric lifting for $T$, then, as in the proof of implication (iv) $\Rightarrow$ (v), $S$ can be lifted to an analytic 2-isometry of the same covariance as $S$. Thus in all assertions (i), (ii), (iv) and (v) we can obtain 2-isometric liftings and dilations for $T$ of the same covariance.

Finally, if $S$ is an analytic 2-isometric lifting for $T$ on $\ka \supset \h$ (as in (ii)), then $\ka_0= \bigvee_{n\ge 0} S^n\h$ is an invariant subspace for $S$ and $S|_{\ka_0}$ is a minimal analytic 2-isometric lifting for $T$ with $\delta_0= {\rm cov}(S|_{\ka_0}) \le {\rm cov}(S)$. However, by the above discussion, we can get an analytic Brownian unitary dilation $\widetilde{S}$ for $T$ with ${\rm cov}(\widetilde{S})=\delta_0$. This ends the proof.
\end{proof}

\begin{corollary}\label{co49}
If $T\in \B$ is an operator similar to a contraction, then $T$ and $T^*$ are restrictions to invariant subspaces of backward shifts on Hilbert spaces of vector-valued analytic functions on $\De$.
\end{corollary}

The following theorem refines the model of A. Olofsson \cite{Ol} for analytic Brownian unitaries. By Theorem \ref{te48}, (v), every operator in $\mathscr{C}_2$ has an analytic Brownian unitary dilation.
\begin{theorem}\label{co-mu}
Let $S\in \mathcal{B}(\ka)$ be an analytic Brownian unitary 
acting on $\ka$.
Let $\mathcal{E}=\n(S^*)$ and let $\mu$ be the positive $\mathcal{B}(\mathcal{E})$-valued measure defined on the $\sigma$-algebra ${\rm Bor}(\mathbb{T})$ of all Borel subsets of $\mathbb{T}$ by
\begin{equation}\label{eq-mu1}
\mu(\sigma)=P_{\mathcal{E}} F(\sigma) \Delta_S|_{\mathcal{E}}, \quad \sigma \in {\rm Bor}(\mathbb{T}),
\end{equation}
where $F$ is the $\mathcal{B}(\R(\Delta_S))$-valued spectral measure on $\mathbb{T}$ of the unitary operator $U$ with $U^*=S^*|_{\R(\Delta_S)}$. Then $S$ is unitarily equivalent to the forward shift on the Dirichlet space $\D(\mu)=\{\Theta k:k\in \ka\}$ of $\mathcal{E}$-valued analytic functions on $\De$, where
\begin{equation}\label{eq-mu2}
(\Theta k)(z)=\sum_{n\ge 0} (P_{\mathcal{E}}S'^{*n}k) z^n, \qquad z\in \De,
\end{equation}
$S'$ being the Cauchy dual operator of $S$,  and the norm on $\D (\mu)$ is induced by $\mu$ as in \cite[(3.1)]{Ol}.

Moreover, if $E_0=\delta^{-1} P_{\n(\Delta_S)} S|_{\R(\Delta_S)}$ and $\delta ={\rm cov}(S)$, then 
\begin{equation}\label{eq244}
\mathcal{E}=\{-\delta^{-1} E_0 U^*d\oplus d: d\in \R(\Delta_S)\}
\end{equation}
and $(1+\delta^2)\delta^{-4} \mu (\sigma)$ is an orthogonal projection, for every $\sigma \in {\rm Bor}(\mathbb{T})$.
\end{theorem}

\begin{proof}
Let $U$ be the unitary operator on $\R(\Delta_S)$ with $U^*=S^*|_{\R(\Delta_S)}$, and let $F:{\rm Bor}(\mathbb{T}) \to \mathcal{B}(\R(\Delta_S))$ be the spectral measure of $U$. Since $\Delta_S= 0\oplus \delta ^2 I$ on $\n(\Delta_S) \oplus \R(\Delta_S)$, where $\delta={\rm cov(S)}$, we have $\Delta_S^{1/2}Sk= \Delta_S^{1/2}Uk=U\Delta_S^{1/2}k$ for $k \in \R(\Delta_S)$. The second relation ensures that $\Delta_S^{1/2}F(\sigma)k=F(\sigma)\Delta_S^{1/2}k$ for $k\in \R(\Delta_S)$. Thus, as in \cite[Lemma 4.1]{Ol}, we can associate to $S$ a positive $\mathcal{B}(\mathcal{E})$-valued measure on $\mathbb{T}$, which in this case has the form \eqref{eq-mu1}, where $\mathcal{E}=\n(S^*)$. Also, corresponding to this measure $\mu$, one can define the Dirichlet space $\D(\mu)$ with the norm induced by $\mu$ as in \cite[(3.1)]{Ol}. By \cite[Theorem 4.1]{Ol} we obtain that $S$ is unitarily equivalent, via the operator $\Theta :\ka \to \D(\mu)$ from \eqref{eq-mu2}, with the forward shift on $\D(\mu)$.

In addition, in this case the subspace $\mathcal{E}$ of $\ka$ can be easily determined and we are able to obtain more information about the measure $\mu$. Indeed, let us consider the canonical matrix representation of the form \eqref{eq12}  of $S$ on $\ka =\n(\Delta_S) \oplus \R(\Delta_S)$, where $V=S|_{\n(\Delta_S)}$ is a shift operator ($S$ being analytic), $Y=U$ is as before, and $E=\delta E_0$, with $E_0$ an isometry from $\R(\Delta_S)$ onto $\n(V^*)$. Then, an element $k=k_0\oplus k_1\in \n(\Delta_S) \oplus \R(\Delta_S)$ belongs to $\mathcal{E}=\n(S^*)$ if and only if $V^*k_0=0$ and $\delta E_0^*k_0 +U^*k_1=0$. Since $\n(V^*)=\R(E_0)$, it follows that $k_0=E_0d_0$ for some element $d_0 \in \R(\Delta_S)$, while the previous equality yields $k_1=- \delta Ud_0$. Considering $d_0$ of the form $d_0 =-\delta^{-1} U^*d$ with $d\in \R(\Delta_S)$, we infer that $\mathcal{E}$ has the form \eqref{eq244}.

Now let $e\in \mathcal{E}$ and $d\in \R(\Delta_S)$ be such that $e=-\delta ^{-1} E_0U^* d\oplus d$. Then 
$$
(1+ \delta^{-2} )d=e\oplus \delta^{-2} (\delta E_0 U^* d \oplus UU^*d)=e \oplus \delta^{-2} S(U^*d).
$$
It follows that 
$$
P_{\mathcal{E}}d= \frac{\delta^2}{1+\delta^2}e=\frac{\delta^2}{1+\delta^2}(-\frac{1}{\delta} E_0 U^* d\oplus d), \quad d\in \R(\Delta_S).
$$     
Thus, for $\sigma \in {\rm Bor}(\mathbb{T})$ and $e\in \mathcal{E}$ as above, we obtain
\begin{eqnarray*}
\mu (\sigma)e &=& P_{\mathcal{E}}F(\sigma)\Delta_Se=\delta^2 P_{\mathcal{E}}F(\sigma)P_{\R(\Delta_S)}e=\delta^2 P_{\mathcal{E}}F(\sigma)d\\
&=& \frac{\delta^4}{1+\delta^2}(-\frac{1}{\delta} E_0 U^*F(\sigma)d \oplus F(\sigma)d)\\
&=& \frac{\delta^4}{1+\delta^2}(-\frac{1}{\delta}E_0 F(\sigma)E_0^* E_0 U^* d \oplus F(\sigma) d)\\
&=&  \frac{\delta^4}{1+\delta^2}(E_0 F(\sigma)E_0^* P_{\n(\Delta_S)}e \oplus F(\sigma)P_{\R(\Delta_S)}e)\\
&=&  \frac{\delta^4}{1+\delta^2}(E_0 F(\sigma)E_0^* \oplus F(\sigma))(e_0 \oplus e_1),  
\end{eqnarray*}
where $e_0=P_{\n(\Delta_S)}e$, $e_1=P_{\R(\Delta_S)}e$. Finally, keeping in mind that $F$ is a spectral measure, so $F(\sigma)^2=F(\sigma)$, we obtain 
\begin{eqnarray*}
\left[ \frac{1+\delta^2}{\delta^4} \mu(\sigma) \right] ^2e &=& (E_0 F(\sigma)E_0^* \oplus F(\sigma) )(E_0 F(\sigma) E_0^* e_0 \oplus F(\sigma) e_1)\\
&=& E_0 F(\sigma) ^2 E_0^* e_0 \oplus F(\sigma) ^2 e_1 =E_0 F(\sigma) E_0 e_0 \oplus F(\sigma) e_1\\
&=& \frac{1+\delta ^2}{\delta^4} \mu (\sigma)e,  
\end{eqnarray*}
for every $\sigma \in {\rm Bor}(\mathbb{T})$ and $e\in \mathcal{E}$. Since $(1+\delta^2) \delta^{-4}\mu(\sigma)$ is a contraction, we conclude that $(1+\delta^2)\delta^{-4} \mu(\sigma)$ is an orthogonal projection for $\sigma \in {\rm Bor}(\mathbb{T})$. This ends the proof. 
\end{proof}

\begin{remark}\label{reMuller}
\rm
V. M\"uller proved in \cite[Corollary 2.3]{M} that if $\h$ is a separable Hilbert space, then, for any operator $T\in \B$, $T$ and $T^*$ are unitarily equivalent with restrictions to invariant subspaces of a backward weighted shift $B_{\alpha}\in \mathcal{B}(\ka)$ where $\ka =\ell_+^2(\h')$ and $\h'$ is a separable Hilbert space. Recall that such an operator $B_{\alpha}$ is defined with respect to a bounded sequence of positive numbers $\alpha =\{\alpha _n\}_{n\ge 1}$ by the relation $B_{\alpha}(\{x_n\}_{n\ge 0})=\{\alpha _nx_n\}_{n\ge 1}$ for all square summable sequences $\{x_n\}_{n\ge 0}\in \ka$ of vectors $x_n \in \h'$. So the operators $T$ and $T^*$ can be lifted to the forward weighted shift $S_{\alpha}=B_{\alpha}^*$ on $\ka$, which is an analytic operator, but $S_{\alpha}$ is not 2-isometric, in general.    
\end{remark}

\section{TWO-ISOMETRIC LIFTINGS OF TYPE I}\label{sect:type0}
\smallskip

\subsection{Type I and type II liftings}
Let $T\in  \B$ and suppose that $S\in \mathcal{B}(\ka)$ is a lifting of $T$ on $\ka \supset \h$. Then the canonical representation of $S$ on $\ka =\h^{\perp} \oplus \h$ has the form
\begin{equation}\label{eq21}
S=
\begin{pmatrix}
W & X\\
0 & T
\end{pmatrix},
\end{equation}
with $W=S|_{\h^{\perp}}$ and $X =P_{\h^{\perp}}S|_{\h}$. In general, the operators $W,X$ and $T$ can be arbitrary,
but they have to satisfy some constraints whenever $S$ belongs to some particular classes of operators. Here we are interested in the case when $S$ is a 2-isometry. In this case, $W$ is also a 2-isometry (as a restriction to an invariant subspace for $S$), and $\Delta_S$ has the form
\begin{equation}\label{eq22}
\Delta_S=
\begin{pmatrix}
\Delta_W & W^*X\\
X^*W & X^*X+\Delta_T
\end{pmatrix}.
\end{equation}
Since necessarily $\Delta_S \ge 0$ and $\Delta_W\ge 0$, one also has $X^*X+\Delta_T\ge 0$. Therefore (see \cite{FF}, \cite{FFK2}), there exists a contraction $\Gamma : \overline{\R(X^*X+\Delta_T)}\to \overline{\R(\Delta_W)}$ such that
\begin{equation}\label{eq23}
W^*X=\Delta_W^{1/2}\Gamma (X^*X+\Delta_T)^{1/2}.
\end{equation}
This relation gives $X^*W|_{\n(\Delta_W)}=0$, which by \eqref{eq22} implies that $\Delta_S|_{\n(\Delta_W)}=0$, hence $\n(\Delta_W)\subset \n(\Delta_S)$. But $\n(\Delta_W)$ is invariant for $W=S|_{\h^{\perp}}$ and so for $S$, while $W|_{\n(\Delta_W)}=S|_{\n(\Delta_W)}$ is an isometry. When $W$ is an isometry on $\h^{\perp}$, one has $\h^{\perp}=\n(\Delta_W)\subset \n(\Delta_S)$. Therefore $W^*X=0$ in this case. Conversely, if $\h^{\perp}\subset \n(\Delta_S)$, then, since $\h^{\perp}$ is invariant for $S$ by \eqref{eq21}, it follows that $W=S|_{\h^{\perp}}$ is an isometry.

On the other hand, we see from \eqref{eq23} that $W^*X=0$ if and only if $\Gamma=0$, because $\R(\Gamma)\subset \overline{\R(\Delta_W)}$. Notice however that, in general, it is difficult to have significant information about the operator $\Gamma$. Using \eqref{eq22}, one has $W^*X=0$ if and only if $\h$ is invariant for $S^*S$. In this case, by \eqref{eq22}, we have $\n(\Delta_S)=\n(\Delta_W) \oplus \n(X^*X+\Delta_T)$. In particular, if $X^*X+\Delta_T=0$, which forces $T$ and $X$ to be contractions (as $D_T^2=X^*X\ge 0$), then $W^*X=0$. In this case it is easy to see that $S$ is an extension of the minimal isometric lifting of $T$, because $\n(\Delta_S) =\n(\Delta_W) \oplus \h$ is an invariant subspace for $S$.

In light of the preceding discussion, we introduce the following definition. 

\begin{definition}\label{def:type}
Using the previous notation, we say that a 2-isometry $S$ of the form \eqref{eq21} on $\ka=\h^{\perp}\oplus \h$ is a \emph{lifting of type I} for $T$ whenever $\h^{\perp}\subset \n(\Delta_S)$. We say that $S$ is a \emph{lifting of type II} for $T$ whenever $\h$ is an invariant subspace for $S^*S$.
\end{definition}

Observe that $S^*S\h \subset \h$ if and only if $W^*X=0$ (in \eqref{eq23}), hence a lifting of type I is also a lifting of type II.

\subsection{A-isometries}

We describe now the operators which have 2-isometric liftings of type I. These operators form a special class of $A$-isometries as follows.

\begin{theorem}\label{te21}
For $T\in \B$ the following statements are equivalent:

\begin{itemize}
\item[(i)] $T$ has a 2-isometric lifting of type I on a Hilbert space containing $\h$;

\item[(ii)] $T$ is either an $A$-isometry for a positive operator $A\neq 0$ on $\h$ with $\Delta_T\le A$, or $T$ is a strongly stable contraction;

\item[(iii)] $T$ has an extension $\widetilde{T}$ on a Hilbert space $\m\supset \h$ with $\widetilde{T}$ of the form
\begin{equation}\label{eq24}
\widetilde{T}=
\begin{pmatrix}
C & \delta E\\
0 & U
\end{pmatrix}
\end{equation}
on a decomposition $\m=\m_0\oplus \m_1$, where $C,E$ are contractions, $U$ is unitary, and $\delta$ is a positive scalar, such that there exist a Hilbert space $\m'$ and isometries $J_0:\D_C\to \m'$, $J_1:\D_E\to \m'$ satisfying the condition
\begin{equation}\label{eq25}
D_CJ_0^*J_1D_E+C^*E=0;
\end{equation}

\item[(iv)] $T$ has a 2-isometric lifting $S$ of the form \eqref{eq21} such that
\begin{equation}\label{eq26}
X^*\Delta_WX+2 {\rm Re}(X^*W^*XT)=0.
\end{equation}
\end{itemize}

Moreover, if these statements are true, then the lifting $S$ of $T$ in (i) and (iv) can be chosen minimal with ${\rm cov}(S)=\|A\|^{1/2}$ for $A$ from (ii), or with ${\rm cov}(S) \le \delta$ for $\delta$ from \eqref{eq24}.
\end{theorem}

\begin{proof}
Let $S\in \mathcal{B}(\ka)$ be a $2$-isometry as in \eqref{eq21}, with $\ka \ominus \h \subset \n(\Delta_S)$. Thus $W=S|_{\ka \ominus \h}$ is an isometry. Then $\Delta _W=0$ and, as we noticed before, one has $W^*X=0$. Therefore, the condition \eqref{eq26} from (iv) is satisfied for such a lifting $S$ of $T$. We obtain that (i) implies (iv).

Assume that $T$ has a 2-isometric lifting $S$ on $\ka \supset \h$ of the form \eqref{eq21} with the operators $W$ and $X$. Then $S^*\Delta_SS=\Delta_S$, and this implies (by  \eqref{eq21} and \eqref{eq22}) the relation
\begin{equation}\label{eq27}
X^*\Delta_WX+2{\rm Re}(X^*W^*XT)+T^*(X^*X+\Delta_T)T=X^*X+\Delta_T.
\end{equation}
Now, if $S$ verifies the condition \eqref{eq26}, then one obtains from \eqref{eq27} that $T$ is an $A$-isometry, where $A=X^*X+\Delta_T\ge 0$ (as $\Delta_S\ge 0$ in \eqref{eq22}), and so $A\ge \Delta_T$. This condition in the case $A=0$ forces $T$ to be a contraction, but in this case $T$ is an $A_0$-isometry with $A_0=s-\lim_{n\to \infty} T^{*n}T^n\ge \Delta_T=-D_T^2$. Here the case $A_0= 0$ corresponds to $T$ being strongly stable, i.e. $\|T^nh\|\to 0$ for $h\in \h$. We have proved that (iv) implies (ii).

Next, let $T$ be an operator as in (ii), i.e. satisfying $T^*AT=A$ with $0\neq A\ge0$ and $A\ge \Delta_T$. Then the operator $A_T:=A-\Delta_T$ is positive and one can suppose $A_T\neq 0$, or in other words, $T$ is not a 2-isometry.

We define the lifting $S_{A,T}$ of $T$ on the space $\h_{A,T}:=\ell_+^2(\overline{\R(A_T)})\oplus \h$ with the block matrix
\begin{equation}\label{eq28}
S_{A,T}=
\begin{pmatrix}
S_+ & \widetilde{A}_T\\
0 & T
\end{pmatrix},
\end{equation}
where $S_+$ is the forward shift on $\ell_+^2(\overline{\R(A_T)})$ and $\widetilde{A}_T=JA_T^{1/2}$, $J$ being the canonical injection of $\overline{\R(A_T)}$ into $\ell_+^2(\overline{\R(A_T)})$. Then, on the above decomposition of $\h_{A,T}$, we have
$$
\Delta_{S_{A,T}}=0\oplus (A_T+\Delta_T)=0\oplus A .
$$
Therefore,
$$
S^*_{A,T}\Delta_{S_{A,T}}S_{A,T}=0 \oplus T^*AT=0\oplus A=\Delta_{S_{A,T}}.
$$
Hence $S_{A,T}$ is a 2-isometry and $S_{A,T}|_{\h_{A,T}\ominus \h}=S_+$ is an isometry. We conclude that $S_{A,T}$ is a 2-isometric lifting of type I for $T$. 

In the case when $T$ is a contraction, it has even an isometric lifting (so of type I and of covariance zero as a $2$-isometry). 
In this case, if $T^*AT=A\neq 0$, then $A_T^{1/2} \ge D_T$. Therefore $\overline{\R(A_T)}=\D_T \oplus (\overline{R(A_T)}\ominus \D_T)$. Using this and \eqref{eq28}, one infers that $S_{A,T}$ is a 2-isometric lifting of type I for $T$ and for the minimal isometric lifting of $T$. But when $T$ is a strongly stable contraction, $T$ cannot be an $A$-isometry with $A\neq 0$, so $A_T=D_T^2$ and $S_{A,T}$ is even the minimal isometric lifting of $T$. Hence (ii) implies (i), and thus the assertions (i), (ii) and (iv) are equivalent.
We also remark that ${\rm cov}(S_{A,T})=\|A\|^{1/2}\neq 0$ and thus ${\rm cov}(S_{A,T})=0=\|A_0\|$, when $T$ is a strongly stable contraction.

Now we show that $S_{A,T}$ is a minimal lifting of $T$, that is, we have $\h_{A,T}=\bigvee_{n\ge0} S^n_{A,T}\h$. To see this, let $k=d\oplus h\in \h_{A,T}$ with $h \in \h$ and $d=\bigoplus_0^{\infty} d_j\in \ell_+^2(\overline{\R(A_T)})$ be such that $k$ is orthogonal to $\bigvee_{n\ge 0} S^n_{A,T}\h$. As $k \perp \h$, it follows that $h=0$. So $k=d\perp S_{A,T}\h= \widetilde{A}_T\h\oplus T\h$. Working in the space $\overline{\R(A_T)}$, we have $d_0\perp A_T^{1/2}\h$, and thus $d_0=0$. By induction, one obtains $d_j=0$ for each $j\ge 1$; consequently $k=d=0$. Thus the minimality condition is true.

The final part of the proof consists in showing the equivalence of (i) with (iii).
Assume that $S$ is a 2-isometric lifting of $T$ on $\ka$ of covariance $\delta$, with $V=S|_{\ka \ominus \h}$ an isometry. Without loss of generality we can assume $\delta>0$ (otherwise, $S$ is an isometry and so $T$ is a contraction, hence $T$ has trivially the form \eqref{eq24}).
Then, by \cite[Theorem 5.80]{AS2}, $S$ has a Brownian unitary extension $\widetilde{S}$ on $\widetilde{\ka}\supset \ka$ of covariance $\delta$ with the canonical representation on $\widetilde{\ka}= \n(\Delta_{\widetilde{S}})\oplus \overline{\R(\Delta_{\widetilde{S}})}$ of the form
$$
\widetilde{S}=
\begin{pmatrix}
\widetilde{V} & \delta \widetilde{E}\\
0 & U
\end{pmatrix}.
$$
Here $\widetilde{V}, \widetilde{E}$ are isometries with $\n(\widetilde{V}^*)=\R(\widetilde{E})$ and $U$ is unitary.

Now, $\widetilde{S}$, as an extension of $S$, has on the decomposition $\widetilde{\ka}=(\ka \ominus \h)\oplus \h \oplus (\widetilde{\ka}\ominus \ka)$ a block matrix of the form
$$
\widetilde{S}=
\begin{pmatrix}
V & \star & \star\\
0 & T & \star\\
0 & 0 & \star
\end{pmatrix}.
$$
Since $\widetilde{S}|_{\ka \ominus \h}=V$ is an isometry and $\widetilde{S}$ is a 2-isometry, it follows that $\h':= \ka \ominus \h \subset \n(\Delta_{\widetilde{S}})$ and $\widetilde{V}|_{\h'}=V$. Inserting $V$ into the matrix representation of $\widetilde{S}$ on $\widetilde{\ka}=\h' \oplus (\n(\Delta _{\widetilde{S}})\ominus \h') \oplus \overline{\R(\Delta_{\widetilde{S}})}$ we get
$$
\widetilde{S}=
\begin{pmatrix}
V & C' & \delta E'\\
0 & C & \delta E\\
0 & 0 & U
\end{pmatrix}.
$$
Here we have represented the isometries $\widetilde{V}$ on $\h'\oplus \m_0$, where $\m_0= \n(\Delta_{\widetilde{S}})\ominus \h'$, and $\widetilde{E}$ from $\m_1= \overline{\R(\Delta_{\widetilde{S}})}$ into $\h'\oplus \m_0$, as
$$
\widetilde{V}=
\begin{pmatrix}
V & C'\\
0 & C
\end{pmatrix},
\quad
\widetilde{E}=
\begin{pmatrix}
E'\\
E
\end{pmatrix},
$$
where $C,C',E$ and $E'$ are contractions satisfying 
$$\widetilde{V}^*\widetilde{E}=V^*E'\oplus (C'^*E'+C^*E)=0 .$$ Also, we have $V^*C'=0$ because $V$ and $\widetilde{V}$ are isometries.

Comparing the two $3\times 3$ matrices of $\widetilde{S}$ we infer
$$
\begin{pmatrix}
T & \star\\
0 & \star
\end{pmatrix}
\begin{bmatrix}
\h\\
\widetilde{\ka}\ominus \ka
\end{bmatrix}
=P_{\widetilde{\ka}\ominus \h'}\widetilde{S}|_{\widetilde{\ka}\ominus \h'}=
\begin{pmatrix}
C & \delta E\\
0 & U
\end{pmatrix}
\begin{bmatrix}
\m_0\\
\m_1
\end{bmatrix},
$$
the two block matrices being given on two different decompositions of $\m:=\widetilde{\ka} \ominus \h'$. Hence $T$ has an extension on $\m=\m_0\oplus \m_1$ of the form \eqref{eq24}. 

It remains now to verify the condition \eqref{eq25}.
Indeed, since $\widetilde{V}$ and $\widetilde{E}$ are isometries, we have $C^*C+C'^{*}C'=I_{\m_0}$ and $E^*E+E'^*E'=I_{\m_1}$. Therefore, $C'^*C'=D_C^2$ and $E'^*E' =D_E^2$.
Since $V^*C'=0$ and $V^*E'=0$, we have $\R(C')\cup \R(E') \subset \n(V^*)=:\m'$. Thus we infer by polar decompositions that $C'=J_0D_C$ and $E'=J_1D_E$, where $J_0,J_1$ are isometries from $\D_C$, respectively from $\D_E$, into $\m'$.
So, $J_0$ and $J_1$ are the canonical mappings of $\D_C=\overline{\R(C'^*)}$ onto $\overline{\R(C')}$, respectively of $\D_E=\overline{\R(E'^*)}$ onto $\overline{\R(E')}$, in the space $\m'$. Finally, the condition $\widetilde{V}^*\widetilde{E}=0$ implies $C'^*E'+C^*E=0$, which becomes $D_CJ_0^*J_1D_E+C^*E=0$, i.e. the condition \eqref{eq25}. We completed the proof that (i) implies (iii).

Conversely, let us assume that $T$ has an extension $\widetilde{T}$ of the form \eqref{eq24} on $\m=\m_0\oplus \m_1$, as in (iii). We denote by $D$ the operator $D=\begin{pmatrix} J_0D_C & J_1D_E \end{pmatrix}$ from $\m_0\oplus \m_1$ into $\m'$. Consider also the operator $J$ from $\m'$ into $\m_2:=\ell_+^2(\overline{\R(D)})$ defined as the canonical injection of $\overline{\R(D)}$ into $\m_2$. Let $\widetilde{D}_C=JJ_0D_C$, $\widetilde{D}_E=JJ_1D_E$, and let $\widehat{S}$ be the (minimal) lifting of $\widetilde{T}$ acting on the space $\widetilde{\m}:= (\m_2 \oplus \m_0)\oplus \m_1$ defined by the block matrix
$$
\widehat{S}=
\begin{pmatrix}
V'& \delta F\\
0 & U
\end{pmatrix}.
$$
Here $V'$ on $\m_2\oplus \m_0$ and $F$ from $\m_1$ into $\m_2\oplus \m_0$ have the representations
$$
V'=
\begin{pmatrix}
S_+ & \widetilde{D}_C\\
0 & C
\end{pmatrix},
\quad F=
\begin{pmatrix}
\widetilde{D}_E\\
E
\end{pmatrix},
$$
where $S_+$ is the forward shift on $\m_2$, while $U, C, E$ and $\delta>0$ are as in \eqref{eq24}. It is clear that $V'$ and $F$ are isometries and that
$$
V'^*F=S_+^*\widetilde{D}_E\oplus (\widetilde{D}_C^*\widetilde{D}_E+C^*E)=0\oplus (D_CJ_0^*J_1D_E+C^*E)=0,
$$
taking into account that $\n(S_+^*)=J\overline{\R(D)}\supset \R(\widetilde{D}_E)$ and using the relation \eqref{eq25}. Thus, one obtains $\Delta_{\widehat{S}}=0\oplus \delta ^2I_{\m_1}$, and we conclude that $\widehat{S}$ is a 2-isometry, in fact even a Brownian isometry with $\delta ^{-2}\Delta_{\widehat{S}}$ an orthogonal projection.

Expressing now $T$ in the matrix of its extension $\widetilde{T}$ on $\m=\h\oplus (\m\ominus \h)$ in the form
$$
\widetilde{T}=
\begin{pmatrix}
T & \star\\
0 & \star
\end{pmatrix},
$$
and then inserting $\widetilde{T}$ in the matrix of $\widehat{S}$ (as a lifting of $\widetilde{T}$), we infer that the subspace $\h':=\m_2 \oplus \h$ is invariant for $\widehat{S}$ and that $S'=\widehat{S}|_{\h'}$ is a 2-isometric lifting for $T$ with $S'|_{\m_2}=S_+$ an isometry. So $T$ satisfies (i). In addition, since $S'^*S'\h \subset \h$, we get $\Delta_{S'}=P_{\h}\Delta_{\widehat{S}}|_{\h}$ and ${\rm cov}(S')\le \delta$. If $\h \cap \m_1\neq \{0\}$, then ${\rm cov}(S')=\delta$, while if $\h \subset \m_0$, then $T$ is a contraction and $S'=V'|_{\h'}$ is an isometry with ${\rm cov}(S')=0$. Notice that, by construction, $S'$ is not minimal. One can consider $S_0:=S'|_{\ka_0}$ where $\ka_0=\bigvee_{n\ge 0}S'^n\h$. Since $\ka_0$ is reducing for $S'$ (see Remark \ref{re311} below), it follows that $S_0$ is a minimal 2-isometric lifting of type I for $T$ with ${\rm cov}(S_0)\le \delta$. Hence (iii) implies (i) and all assertions are now proved. 
\end{proof}

It is easy to see that every operator $T$ of the form \eqref{eq24} on $\h=\h_0 \oplus \h_1$ is a $P_{\h_1}$-isometry. In addition, one has
$$
\delta ^2 P_{\h_1} -\Delta_T=
\begin{pmatrix}
D_C^2 & -\delta C^*E\\
-\delta E^*C & \delta ^2 D_E^2
\end{pmatrix}.
$$
Thus, when the condition \eqref{eq25} holds, i.e. $-C^*E=D_CJ_0^*J_1D_E$, the above matrix is positive. Hence $\delta^2 P_{\h_1} \ge \Delta_T$. Applying for such an operator $T$ the arguments used in the proof of equivalences of (i) with (ii) and (iii), we deduce the following result.

\begin{corollary}\label{co22}
For $T\in \B$ the following statements are equivalent:

\begin{itemize}
\item[(i)] $T$ has the form \eqref{eq24} and the condition \eqref{eq25} holds true;
\item[(ii)] $T$ is a $P$-isometry for an orthogonal projection $P$ with $\delta^2 P\ge \Delta_T$ and some scalar $\delta >0$;
    \item[(iii)] $T$ has a (minimal) Brownian isometric lifting $S$ of type I, with $\delta^{-2} \Delta_S$ an orthogonal projection and $\delta={\rm cov}(S)$.
\end{itemize}
\end{corollary}

An application of Theorem \ref{te21} concerns the $(A,2)$-{\it expansive} operators studied in  \cite{JKKL}, which in fact are the $\Delta_A(T)$-contractions, where $\Delta_A(T)=T^*AT-A$ and $A$ is a positive operator. We obtain the following consequence.

\begin{corollary}\label{co230}
Let $T\in \B$ be an invertible $(A,2)$-expansive operator such that $A\ge \Delta_T$ or $A\ge \Delta_{T^{-1}}$. Then $T$, respectively $T^{-1}$, have minimal 2-isometric liftings of type I.
\end{corollary}

\begin{proof}
From hypothesis and \cite[Theorem 3.10 (ii)]{JKKL}, it follows that $T^*AT=A$, or, equivalently, $(T^*)^{-1}AT^{-1}=A$. Thus one can apply Theorem \ref{te21} to $T$ if $A\ge \Delta_T$, respectively to $T^{-1}$ if $A\ge \Delta_{T^{-1}}$, to obtain the conclusion.
\end{proof}

\begin{remark}\label{re23}
\rm
Theorem \ref{te21} can be seen as a generalization of the well-known theorem of isometric lifting of a contraction (see \cite{FF, SFb}). The operator $S_{A,T}$ in \eqref{eq28} is a minimal 2-isometric lifting of $T$ on the space $\ka= \ell_+^2(\overline{\R(A_T)})\oplus \h$, while in the case when $T$ is a contraction (corresponding to $A=0$ in (ii)), the operator $S_{0,T}$ is the minimal isometric lifting of $T$. But the 2-isometric lifting $S_{A,T}$ with $A \neq 0$ is not uniquely determined by the minimality condition, up to unitary equivalence which fixes $\h$, as we will see in the next section. This happens even for contractions $T$ with $\widehat{A}:=s-\lim_{n\to \infty} T^{*n}T^n\neq 0$ when we consider the corresponding lifting $S_{\widehat{A},T}$. In this case $T^*\widehat{A}T=\widehat{A} \ge \Delta_T=-D_T^2$. Then $\widehat{A}_T=\widehat{A}+D_T^2$ and $S_{\widehat{A},T}$ is a minimal 2-isometric lifting of $T$ of covariance $\|\widehat{A}\|^{1/2}>0$, so $S_{\widehat{A},T}$ is not isometric. Hence $S_{\widehat{A},T}$ cannot be unitarily equivalent to $S_{0,T}$.
\end{remark}

\begin{remark}\label{re24}
\rm

The condition \eqref{eq25} does not imply the condition $C^*E=0$, in general, but the converse implication holds by choosing $J_0, J_1$ such that the subspaces $J_0\D_C$ and $J_1\D_E$ are orthogonal in $\m'$.
Therefore, the operators $T$ of the form \eqref{eq24} with the condition $C^*E=0$ form a special class of operators with Brownian isometric liftings of type I having the covariance operators the scalar multiples of orthogonal projections (by Corollary \ref{co22}). We mention here two special cases. If $C$ is a coisometry in \eqref{eq24}, then \eqref{eq25} implies $E=0$, hence the operator $C \oplus U$ is a contraction. Also, if $C$ or $E$ are isometries, then the two conditions \eqref{eq25} and $C^*E=0$ are simultaneously satisfied.
\end{remark}

\subsection{Expansive operators}

The case when $C$, the upper left entry of the matrix \eqref{eq24}, is an isometry is related to expansive operators.  Recall that $T$ is said to be expansive if $T^*T \ge I$.

\begin{theorem}\label{te25}
For $T\in \B$ the following statements are equivalent :
\begin{itemize}
\item[(i)] $T$ is expansive and has a minimal 2-isometric lifting of type I;

\item[(ii)] $T$ is an $A$-contraction for some $A\in \B$ such that $A\ge \Delta_T\ge 0$;

\item[(iii)] $T$ has an extension $\widetilde{T}$, defined on a Hilbert space $\ka$, which on an orthogonal decomposition $\ka= \ka_0 \oplus \ka _1$ has the form
 \begin{equation}\label{eq29}
 \widetilde{T}=
 \begin{pmatrix}
 V & \delta E\\
 0 & U
 \end{pmatrix},
 \end{equation}
 with $V$ an isometry on $\ka_0$, $U$ a unitary operator on $\ka_1$, $E$ a contraction from $\ka_1$ into $\ka _0$ satisfying $V^*E=0$, and $\delta \ge \|\Delta_T\|^{1/2}$.
\end{itemize}

Moreover, if these statements are true, then one can chose $\widetilde{T}$ in \eqref{eq29} with $\delta =\|A\|^{1/2}$ for $A$ as in (ii).
\end{theorem}

\begin{proof}
We may assume that $T$ is non-isometric, i.e. $\Delta_T\neq 0$.
It is clear that (i) implies (ii) by Theorem \ref{te21}. Conversely, let us assume that $T^*AT\le A$ with $A \ge \Delta _T \ge 0$, as in (ii). Then there exists a contraction $\widehat{T}$ on $\overline{\R(A)}$ such that $\widehat{T}A^{1/2}h=A^{1/2}Th$ for $h\in \h$. In fact, defining the operator $A_0: \h \to \overline{\R(A)}$ by $A_0h=A^{1/2}h$, $h\in \h$, we have $A_0T=\widehat{T}A_0$.

Let $\widehat{V}$ on $\widehat{\h} \supset \overline{\R(A)}$ be the minimal isometric lifting of $\widehat{T}$. Then, as $T^*T\ge I$, by the Treil-Volberg generalization of the commutant lifting theorem (see  \cite{BFF}, \cite{FFK2}, \cite{TV}), there exists an operator $\widehat{A}\in \mathcal{B}(\h, \widehat{\h})$ with $\|\widehat{A}\|=\|A_0\|$ such that
$$
P_{\widehat{\h}, \overline{\R(A)}}\widehat{A}=A_0, \quad \widehat{A}T=\widehat{V}\widehat{A}.
$$
Defining $B\in \B$ by $B=\widehat{A}^*\widehat{A}$, we have $T^*BT=B$, i.e. $T$ is a $B$-isometry with $\|B\|=\|A\|$ and $B \ge A_0^*A_0=A \ge \Delta_T\ge 0$.
Hence (ii) implies (i) by Theorem \ref{te21}.

To show that (i) implies (iii) we can suppose (by Theorem \ref{te21}, (i)) that $T$ is a $B$-isometry with $B\ge \Delta_T \ge 0$. Then there exists an isometry $W$ on $\overline{\R(B)}$ satisfying the relation $WB^{1/2}=B^{1/2}T$.

Let us denote $B_0=\|B\|^{-1}B$. Since $T$ is expansive and a $B_0$-isometry, we have
$$
T^*(I-B_0)T=T^*T-B_0\ge I-B_0,
$$
and then, by Douglas' lemma  \cite{D}, one obtains a contraction $C$ on $\h$ satisfying the relation
$$
(I-B_0)^{1/2}=T^*(I-B_0)^{1/2}C.
$$
This, together with the relation $B\ge \Delta_T$, yield the inequality
$$
T^*(I-B_0)^{1/2}(I-CC^*)(I-B_0)^{1/2}T=T^*T-I \le B= \|B\|B_0.
$$

Consider $V$ on $\ka_0\supset \h$ to be the minimal isometric lifting of $C$. Therefore $P_{\ka_0, \h}V=CP_{\ka_0,\h}$. Define the linear operator $E_0: \R(B^{1/2})\to \ka_0$ by the relation
$$
E_0(B_0^{1/2}h)=(I-VV^*)(I-B_0)^{1/2}Th, \quad h\in \h.
$$
Since $V^*|_{\h}=C^*$, from the above relations one obtains
\begin{eqnarray*}
\|E_0B_0^{1/2}h\|^2 &=& \langle (I-VV^*)(I-B_0)^{1/2}Th, (I-B_0)^{1/2}Th \rangle\\
&=& \langle (I-CC^*)(I-B_0)^{1/2}Th, (I-B_0)^{1/2}Th\rangle \\
&=& \langle \Delta_Th,h \rangle \le  \langle Bh,h\rangle.
\end{eqnarray*}
Setting $E_1:= \delta^{-1} E_0$, where $\delta= \|B\|^{1/2}\ge \|\Delta_T\|^{1/2}$, the previous inequality becomes
$$
\|E_1B^{1/2}h\| \le \|B^{1/2}h\|, \quad h\in \h .
$$
So $E_1$ can be continuously extended to a contraction, also denoted $E_1$, from $\overline{\R(B)}$ into $\ka_0$. In addition, by the definition of $E_0$, one has $V^*E_1=0$.

Let now $U$ be the minimal unitary extension of the above isometry $W$ on 
$$\ka_1 =\bigvee _{n\ge 0} U^{*n}\overline{\R(B)},$$ and let $E:\ka_1 \to \ka_0$ be a contractive extension of $E_1$ to $\ka_1$, for example $E=E_1P$ where $P=P_{\ka_1, \overline{\R(B)}}$ is the projection of $\ka_1$ onto $\overline{\R(B)}$. Clearly, one has $V^*E=0$.

Consider the Hilbert space $\ka=\ka_0 \oplus \ka_1$ and the operator $\widetilde{T}\in \mathcal{B}(\ka)$ having the block matrix \eqref{eq29} with the above operators $V,U,E$ and number $\delta$. Define now the operator $Z:\h \to \ka$ with $\R(Z)\subset \ka_0 \oplus \overline{\R(B)}$, by the relation
$$
Zh=(I-B_0)^{1/2}h\oplus B_0^{1/2}h, \quad h\in \h.
$$
Obviously, $Z$ is an isometry. We now show that $ZT=\widetilde{T}Z$. We have $U|_{\overline{\R(B)}}=W$ and $E|_{\R(B)}=E_1$, so
$$
ZTh=(I-B_0)^{1/2}Th\oplus B_0^{1/2}Th=(I-B_0)^{1/2}Th \oplus UB_0^{1/2}h.
$$
On the other hand, using \eqref{eq29}, we have
$$
\widetilde{T}Zh=[V(I-B_0)^{1/2}h \oplus E_0B_0^{1/2}h]\oplus UB_0^{1/2}h.
$$
But $V^*|_{\h}=C^*$, and from the definitions of $E_0$ and $C$ we get
\begin{eqnarray*}
(I-B_0)^{1/2}Th&=& VV^*(I-B_0)^{1/2}Th \oplus (I-VV^*)(I-B_0)^{1/2}Th\\
&=& VC^*(I-B_0)^{1/2}Th\oplus E_0B_0^{1/2}h=V(I-B_0)^{1/2}h \oplus E_0B_0^{1/2}h.
\end{eqnarray*}
Hence $ZT=\widetilde{T}Z$, which means that $T$ is unitarily equivalent to the operator $T'=\widetilde{T}|_{\R(Z)}$ and $\widetilde{T}$ is an extension of $T'$ of the form \eqref{eq29}.

Identifying $T$ with $T'$, it follows that $T$ has the property (iii), and we conclude that (i) implies (iii). In addition, we note that $\delta= \|B\|^{1/2}$. We have seen that such an operator $B$ can be induced by an operator $A$ satisfying (ii) with $\|B\|=\|A\|$, so we can choose $\delta=\|A\|^{1/2}\ge \|\Delta_T\|^{1/2}$ (as $\Delta_T\ge 0$) in this case.

To complete the proof we show that (iii) implies (i). Indeed, let us assume that $\widetilde{T}$ on $\m=\m_0 \oplus \m_1\supset \h$ is an extension of $T$, $\widetilde{T}$ having the form \eqref{eq29} with $V^*E=0$. Then $\h$, as a subspace of $\m$, is invariant for $\widetilde{T}$. But $\widetilde{T}$ is expansive because $\Delta_{\widetilde{T}}=0\oplus \delta^2 E^*E\ge 0$, and we infer that $\Delta_T= P_{\h}\Delta_{\widetilde{T}}|_{\h}\ge 0$. Thus $T$ is also expansive. Furthermore, as we have seen before in Corollary \ref{co22}, $\widetilde{T}$ is a $P_{\m_1}$-isometry and $\delta ^2P_{\m_1}\ge \Delta _{\widetilde{T}}$. Now $P_{\m_1}\h\neq \{0\}$, because otherwise one has $\h\subset \m_0$, so $T=\widetilde{T}|_{\h}=V|_{\h}$. Hence $\h$ is invariant for $V$ and $T$ is an isometry, in contradiction with our assumption from the beginning of the proof. Therefore $P_{\m_1}|_{\h}\neq 0$, which implies
$$
P_{\h}P_{\m_1}|_{\h}\ge \delta^{-2} P_{\h}\Delta_{\widetilde{T}}|_{\h}=\delta^{-2}\Delta_T.
$$
Thus $P_{\h}P_{\m_1}|_{\h} \neq 0$ (as $\Delta_T \neq 0$). Finally, from the relation $\widetilde{T}^*P_{\m_1}\widetilde{T}=P_{\m_1}$, we deduce that $T^*A_1T=A_1$, where $A_1= P_{\h} P_{\m_1}|_{\h}$. Hence $T$ is a $\delta ^2A_1$-isometry with $\delta^2A_1\ge \Delta_T$. Therefore, $T$ satisfies (by Theorem \ref{te21}) the requirements of (i). This proves that (iii) implies (i). The proof is  complete.
\end{proof}

Let us note that in the proof of implication (i) $\Rightarrow$ (iii) we have used an argument inspired from \cite[Theorem 5.80]{AS2}, which concerns $2$-isometries. In the case when $T$ is a 2-isometry one can choose in the above proof $B=\Delta_T$. This leads to the isometry $E_1$ and one can consider $E$ an isometric extension of $E_1$ with $\R(E)=\n(V^*)$. Then, in this case, $\widetilde{T}$ is a Brownian unitary extension of $T$ with ${\rm cov}(\widetilde{T})=\delta={\rm cov}(T)$, and so one recovers the result of \cite{AS2}.

Concerning the operator $\widetilde{T}$ in \eqref{eq29}, we remark that we do not assume any relationship between the operators $E$ and $U$. However, the operator $\widetilde{T}$ satisfy the equivalent conditions of Theorem \ref{te25} and they can be described by the special Brownian isometric liftings from Corollary \ref{co22}, as follows.

\begin{proposition}\label{pr28}
For $T$ on $\h$ the following statements are equivalent:
\begin{itemize}
\item[(i)] $T$ has a block matrix decomposition \eqref{eq29} on $\h=\h_0\oplus \h_1$, with $\h_0$ an invariant subspace for $T^*T$;
\item[(ii)] $T$ is an expansive $P$-isometry for an orthogonal projection $P$ with $\delta^2P\ge \Delta_T$ and some scalar $\delta>0$;
\item[(iii)] $T$ has a (minimal) Brownian isometric lifting $S$ of type I, on a space $\ka=\h^{\perp} \oplus \h$ such that $\n(\Delta_S)\ominus \h^{\perp}$ is an invariant subspace for $S$, $\delta^{-2}\Delta_S$ is an orthogonal projection and $\delta={\rm cov}(S)$.
\end{itemize}
\end{proposition}

\begin{proof}
Suppose that $T$, acting on $\h=\h_0\oplus \h_1$, has the form \eqref{eq29} with the block matrix given by the operators $V,E$ and $U$, with $V^*E=0$ and $\delta>0$. Obviously, the condition $V^*E=0$ in \eqref{eq29} means that $\h_0$ reduces $T^*T$. Proceeding as in the proof of Theorem \ref{te21} (the construction of $\widehat{S}$), since $V=T|_{\h_0}$ is an isometry in \eqref{eq29}, we find a lifting $S$ of $T$ on $\ka=\ell_+^2(\widetilde{\D}_E)\oplus \h_0 \oplus \h_1$ with the representations
$$
S=
\begin{pmatrix}
S_+ & 0 & \delta \widetilde{D}_E\\
0 & V & \delta E\\
0 & 0 & U
\end{pmatrix}
=\begin{pmatrix}
\widetilde{V} & \delta \widetilde{E}\\
0 & U
\end{pmatrix}
= \begin{pmatrix}
S_+ & \star\\
0 & T
\end{pmatrix},
$$
where $S_+$ is the forward shift on $\h^{\perp}=\ka \ominus \h$, $\widetilde{V}=S_+\oplus V$ on $\h^{\perp} \oplus \h_0$, $\widetilde{E}= \begin{pmatrix} \widetilde{D}_E & E \end{pmatrix}^{\rm tr}: \h_1 \to \h^{\perp} \oplus \h_0$, while $\widetilde{D}_E=J_ED_E$ with $J_E$ the canonical injection of $\D_E$ into $\h^{\perp}$. Since $\widetilde{V}$ and $\widetilde{E}$ are isometries and $V^*E=0$, it follows that $S$ is a Brownian isometric lifting of type I for $T$, with $\n(\Delta_S)=\h^{\perp} \oplus \h_0$ and $\n(\Delta_S)\ominus \h^{\perp}=\h_0$ an invariant subspace for $S$. Remark also that $\delta ^{-2} \Delta_S =P_{\h_1}$, so $\delta = {\rm cov}(S)$ and, in addition, that $S$ is a minimal lifting of $T$. Hence (i) implies (iii).

Conversely, let us assume that such an operator $S$ on $\ka =\h^{\perp} \oplus \h$ is a Brownian isometric lifting for $T$ with $\h^{\perp} \subset \n(\Delta_S)$ (so $S$ is of type I) and that $\n(\Delta_S) \ominus \h^{\perp}$ is an invariant subspace for $S$. Then the operators $V_0=S|_{\h^{\perp}}$ and $V_1= S|_{\n(\Delta_S)\ominus \h^{\perp}}$ are isometries. Hence the isometry $\widetilde{V}:=S|_{\n(\Delta_S)}$ can be written as the direct sum $\widetilde{V}=V_0\oplus V_1$ on $\h^{\perp} \oplus (\n(\Delta_S)\ominus \h^{\perp})$ (as both subspaces are invariant for $\widetilde{V}$). Then, from the block matrix of $S$ on $\ka =\n(\Delta_S) \oplus \overline{\R(\Delta_S)}$ with the operators $\widetilde{V}, \widetilde{E}, U$ (as above) and $\delta ={\rm cov} (S)$, it follows that $T$ has on $\h=(\n(\Delta_S)\ominus \h^{\perp})\oplus \overline{\R(\Delta_S)}=:\h_0 \oplus \h_1$ the representation
$$
T=P_{\h}S|_{\h}=
\begin{pmatrix}
V_1 & \delta F\\
0 & U
\end{pmatrix},
\quad F=P_{\h_0}\widetilde{E}|_{\h_1}.
$$
Clearly, $\delta={\rm cov} (S)>0$ because $T$ (like $S$) can be assumed non-isometric. Since $\widetilde{V}^*\widetilde{E}=0$ in the block matrix of the 2-isometry $S$, we infer $V_1^*F=0$. Hence $T$ has a representation of the form \eqref{eq29} on $\h=\h_0 \oplus \h_1$, where the subspace $\h_0$ reduces $T^*T$. Therefore (iii) implies (i).

Now, since the operators of the form \eqref{eq29} are expansive and they satisfy \eqref{eq24} with the condition \eqref{eq25}, the implication (i) $\Rightarrow$ (ii) follows from Corollary \ref{co22}. Conversely, if $T$ satisfies the assertion (ii), then $T$ has the form \eqref{eq24} on $\h=\h_0\oplus \h_1$ with the condition \eqref{eq25}, by Corollary \ref{co22}. The operator $T$ being expansive by (ii), it follows that $C=T|_{\h_0}$ is expansive, hence $C$ is an isometry ($C$ being a contraction by \eqref{eq25}). Then the condition \eqref{eq25} reduces to $C^*E=0$, where $E=P_{\h_0}T|_{\h_1}$. We conclude that $T$ has the form \eqref{eq29}. Hence (ii) implies (i).
\end{proof}

\begin{remark}\label{re29}
\rm
The Brownian isometric liftings mentioned in Corollary \ref{co22} and Proposition \ref{pr28} are not Brownian unitaries, because $\n(\widetilde{V}^*)\neq \R(\widetilde{E})$ (in the previous proof), in general. While both satisfy the conditions (i), (ii) and (iii) of \cite[Theorem 5.20]{AS2}, they not satisfy the condition (iv) of  the same result \cite[Theorem 5.20]{AS2} in which Brownian unitaries are characterized.
\end{remark}

\subsection{Concave operators and operators similar to isometries}
In the case of concave operators we can say more.
\begin{theorem}\label{pr210}
For an operator $T\in \B$ the following statements hold:
\begin{itemize}
\item[(i)] $T$ is concave if and only if $T$ has an extension $T_1$ on $\m=\m_0\oplus \m_1 \supset \h$ of the form
\begin{equation}\label{eq210}
T_1=
\begin{pmatrix}
V & \delta E_1\\
0 & W
\end{pmatrix},
\end{equation}
where $V,W$ are isometries, $E_1$ is a contraction with $V^*E_1=0$, $W^*E_1^*E_1W\le E_1^*E_1$, and $\delta =\|\Delta_T\|^{1/2}$. In this case $T_1$ is also concave with $\|\Delta_{T_1}\|^{1/2}=\delta$.

\item[(ii)] If $T$ is concave, and one of the sequences $\{\Delta_T^{1/2}T^n\}$ or $\{\Delta_{T_1}^{1/2}T_1^n\}$ converges strongly to zero, then the other also strongly converges to zero. In this case the sequences $\{\frac{1}{\sqrt{n}}T^n\}$ and $\{\frac{1}{\sqrt{n}}T_1^n\}$ converge strongly to zero.

\item[(iii)] A concave operator $T$ has a minimal $2$-isometric lifting of type I and of covariance $\|\Delta_T\|^{1/2}$.
\end{itemize}
\end{theorem}

\begin{proof}
Assume $T$ concave. By applying Theorem \ref{te25} with $A=\Delta_T$, and using the same notation from its proof, one obtains an extension $T_1$ of $T$ on the space $\m=\ka_0 \oplus \overline{\R(B)}$ of the form \eqref{eq210}. More precisely, $V,W$ are isometries on $\ka_0$, respectively on $\overline{\R(B)}$, $E_1$ is a contraction from $\overline{\R(B)}$ into $\ka_0$ with $V^*E_1=0$, and $\delta=\|\Delta_T\|^{1/2}=\|B\|^{1/2}$. As $T$ is concave, $T_1$ will be concave, too. Indeed, the representation \eqref{eq210} of $T_1$ implies
$$
\Delta_{T_1}=0\oplus \delta^2E_1^*E_1, \quad T_1^*\Delta_{T_1}T_1=0\oplus \delta^2 W^*E_1^*E_1W,
$$
and from the proof of Theorem \ref{te25} we have, for $h\in \h$,
\begin{eqnarray*}
\|E_1WB^{1/2}h\|^2&=& \|E_1B^{1/2}Th\|^2= \langle T^*\Delta_TTh,h\rangle\\
&\le & \langle \Delta_Th,h \rangle=\|E_1B^{1/2}h\|^2.
\end{eqnarray*}
Therefore $W^*E_1^*E_1W\le E_1^*E_1$, or equivalently $T_1^*\Delta_{T_1}T_1\le \Delta_{T_1}$. Thus $T_1$ is concave. It is clear that $\|\Delta_{T_1}\|^{1/2} \le \delta$. It follows from the proof of Theorem \ref{te21} that one can consider $Z$ an isometry from $\h$ into $\m$ such that $ZT=T_1Z$. This implies $\Delta_T =Z^*\Delta _{T_1}Z$, so $\delta =\|\Delta _T \|^{1/2} \le \|\Delta _{T_1}\|^{1/2}$. Hence  $\|\Delta_{T_1}\|^{1/2}=\delta$.
The direct implication of the assertion (i) is proved, while the converse part is easy (that is, if $T$ has an extension $T_1$ of the form \eqref{eq210}, then $T=T_1|_{\h}$ is concave because $\h$ is invariant for $T_1$). 

To show the assertion (ii) we remark that if
$$
\|\Delta_{T_1}^{1/2}T_1^n(k_0\oplus k_1)\|=\delta \|E_1W^nk_1\|\to 0, \quad n\to \infty ,
$$
for all $k_0 \in \ka_0$, $k_1 \in \overline{\R(B)}$, then, for $k_1=B^{1/2}h$ with $h\in \h$, we have
$$
\|
\Delta_T^{1/2}T^nh\|=\|E_1W^nB^{1/2}h\| \to 0, \quad n\to \infty.
$$
Conversely, if this last convergence holds, then, by the Banach-Steinhaus theorem ($E_1$ and $W$ being contractions), one has $E_1W^nk_1\to 0$ for any $k_1\in \overline{\R(B)}$. In other words, $\Delta_{T_1}^{1/2}T_1^n\to 0$ strongly. So, if one of the sequences $\{\Delta_T^{1/2}T^n\}$ or $\{\Delta_{T_1}^{1/2}T_1^n\}$ converges strongly to zero, then the other also strongly converges to zero. 
If this happens, then, by a similar argument to that used in the proof of  \cite[Theorem 3.10]{JKKL}, one can show that $\frac{1}{\sqrt{n}}T^n\to 0$ and $\frac{1}{\sqrt{n}}T_1^n\to 0$ strongly. So (ii) holds.

The assertion (iii) follows immediately. Thus, considering $T$ a $B$-isometry (by Theorem \ref{te25}) with $B\ge \Delta_T$ and $\|B\|=\|\Delta_T\|$, one can construct the 2-isometry $S_{B,T}$ as in \eqref{eq28}. This operator is a minimal lifting of $T$ of type I and of covariance $\|\Delta_T\|^{1/2}$.
\end{proof}

\begin{remark}\label{re211}
\rm
If $T$ is a concave operator, then, by Theorem \ref{te25}, $T$ has an extension $\widetilde{T}$ of the form \eqref{eq29} which, in fact, can be obtained from $T_1$ in \eqref{eq210} by extending $W$ to a unitary operator. Notice that $\widetilde{T}$ is not concave, in general.
So Theorem \ref{pr210} shows that the appropriate extensions describing concave operators are those of the form \eqref{eq210}. This provides a model for the concave operators $T$ with $\|\Delta_T \|\le \delta ^2$, for some fixed $\delta>0$. A related fact is given by the last assertion of (i) in Theorem \ref{pr210}, which says that a concave $T$ is of class $C_{0\cdot}$ (as a $\Delta_T$-contraction) if and only if the corresponding concave model $T_1$ is of the same class $C_{0\cdot}$ (as a $\Delta_{T_1}$-contraction). We refer to \cite{SFb} for details about the class $C_{0\cdot}$.

Let us remark that $T_1$ in \eqref{eq210} is different from the Brownian extension $T_b$ of a concave $T$ obtained in  \cite[Theorem B]{McC}, when $\|T\|\le \sqrt{2}$ (i.e. $\Delta_T\le I$). Indeed, in this case $\Delta_{T_b}$ is an orthogonal projection, contrary to $\Delta_{T_1}$ (in general).
\end{remark}

Note also that the extension $T_1$ from \eqref{eq210} of a concave operator $T$ is an improved version of the extensions obtained in \cite[Proposition 2.2, Theorem 2.1]{BS2}, where different 2-isometric liftings have been directly constructed.

Theorem \ref{te21} applies in particular to operators similar to isometries (these are $A$-isometries with $A$ invertible). If $T$ is such an operator satisfying $T^*AT=A$ with $A\ge \beta I$ for a scalar $\beta >0$, then $A\ge \beta T^*T\ge \beta \Delta_T$ and $T$ is also an $A_{\beta}$-isometry, where $A_{\beta}=\beta^{-1}A$. In addition, since $A_{\beta}-\Delta_T\ge T^*T-\Delta_T=I$, one has $\overline{\R(A_{\beta}-\Delta_T)}=\h$. Hence the corresponding lifting $S_{A_{\beta},T}$ acts on $\ell_+^2(\h)$.

Another interesting case is that of \emph{quasi-isometries}, i.e. the $T^*T$-isometries, where $A=T^*T$ is not necessary invertible in this case. 

\begin{corollary}\label{co214}
If $T\in \B$ is similar to an isometry, or $T$ is a quasi-isometry, then $T$ has a minimal 2-isometric lifting of type I on the space $\ell_+^2(\h)$. Such a lifting is the operator $S_{A_{\beta},T}$ (respectively $S_{T^*T,T}$), having covariance $\|A_{\beta}\|^{1/2}$ (respectively $\|T\|$).
\end{corollary}
Notice that a quasi-isometry $T$ is similar to an isometry if and only if $T$ is injective with $\R(T)$ closed (see \cite{MS}). Also, a quasi-isometry $T$ is expansive if and only if $V=T|_{\overline{\R(T)}}$ is an isometry, $X=P_{\overline{\R(T)}}T|_{\n(T^*)}$ is expansive and $V^*X=0$. Hence Theorem \ref{te25} cannot be applied to quasi-isometries, in general.

A result obtained in \cite[Theorem 1.1]{BS1} shows that an expansive operator is similar to an isometry if and only if the Ces\`aro means $M_n:=\frac{1}{n+1}\sum_{j=0}^n T^{*j}T^j$ are bounded. In this case the sequence $\{M_n\}$ strongly converges to an operator $M$ such that $T^*MT=M$ and $M\ge T^*T$. Thus, the expansive operators $T$ which are similar to isometries have 2-isometric liftings of type I. These liftings can be chosen of covariance equal to $\|M\|^{1/2}$ by Corollary \ref{co214}. 

\section{TWO-ISOMETRIC LIFTINGS OF TYPE II}\label{sect:type1}

\subsection{A-contractions}

As in Definition \ref{def:type}, a 2-isometric lifting $S$ for $T\in \B$ is of type II if $\h$ is invariant for $S^*S$. The following result should be compared to Theorem \ref{te21} for the type I case.

\begin{theorem}\label{te31}
For $T\in \B$ the following statements are equivalent:

\begin{itemize}
\item[(i)] $T$ has a 2-isometric lifting of type II on a Hilbert space containing $\h$;

\item[(ii)] $T$ is an $A$-contraction for a positive operator $A\neq 0$ on $\h$ with $\Delta_T \le A$;

\item[(iii)] $T$ has a lifting $T_*$ on $\h_*\supset \h$ which is a $B$-isometry for some $B\in \mathcal{B}(\h_*)$ with $0\neq B\ge 0$, such that $\Delta _{T_*}\le B$ and $B\h\subset \h$;

\item[(iv)] $T$ has a 2-isometric lifting $S$ of the form \eqref{eq21} with operators $W$ and $X$ satisfying the condition
\begin{equation}\label{eq31}
X^*\Delta_WX + 2 {\rm Re }(X^*W^*XT)\ge 0.
\end{equation}
\end{itemize}

Moreover, if these statements are true, then the lifting $S$ in (i) and (iv) can be chosen either minimal, or with $\Delta_{S|_{\h ^\perp}}=2P$, where $P$ is an orthogonal projection, with shifts on the main diagonal of $S|_{\h^{\perp}}$ in its canonical representation and with ${\rm cov}(S)=\sqrt{2}\cdot {\rm max}\{1, \|A\|^{1/2}\}$ for $A$ from (ii).
\end{theorem}

\begin{proof}
If $S$ on $\ka \supset \h$ is a 2-isometric lifting of type II of $T$ with the block matrix \eqref{eq21}, then $W^*X=0$ and so the condition \eqref{eq31} is satisfied. Hence (i) implies (iv). Next, if a 2-isometric lifting $S$ of the form \eqref{eq21} satisfies \eqref{eq31}, then, from the relation \eqref{eq27}, we infer that $T^*AT\le A$, where $A=X^*X+\Delta_T\ge \Delta_T$. Clearly, one has $A\ge 0$ since $\Delta_S\ge 0$ in \eqref{eq22}. Notice that if $A=0$, then $I-T^*T=X^*X\ge 0$. Therefore $T$ is a contraction, i.e. an $I$-contraction with $I\ge \Delta_T$. Hence one can always consider $A\neq 0$, and thus we proved that (iv) implies (ii).

Now, let $T$ be non-isometric and let $A$ be as in (ii), i.e. satisfying $T^*AT\le A$ with $A\ge 0$ and $A\ge \Delta_T\neq 0$. Then there exists a contraction $\widehat{T}$ on $\overline{\R(A)}$ such that $\widehat{T}A^{1/2}=A^{1/2}T$. Let $\widehat{S}$ be the forward shift on $\ell_+^2(\D_{\widehat{T}})$ and set $D=JD_{\widehat{T}}$, where $J$ is the canonical injection of $\D_{\widehat{T}}$ into $\ell_+^2(\D_{\widehat{T}})$. Define the lifting $T_*$ of $T$ and the extension $\widehat{A}$ of $A$ on the space $\h_*:= \ell_+^2(\D_{\widehat{T}}) \oplus \h$ by the block matrices
\begin{equation}\label{eq32}
T_*=
\begin{pmatrix}
\widehat{S} & DA^{1/2}\\
0 & T
\end{pmatrix},
\quad \widehat{A}=
\begin{pmatrix}
I & 0\\
0 & A
\end{pmatrix}.
\end{equation}
We have $\widehat{S}^*DA^{1/2}=0$. We use this relation to obtain that
\begin{eqnarray*}
T_*^*\widehat{A}T_*&=& I \oplus (A^{1/2}D_{\widehat{T}}^2A^{1/2}+T^*AT)= I \oplus (A-A^{1/2}\widehat{T}^*\widehat{T}A^{1/2}+T^*AT)\\
&=& I \oplus A=\widehat{A}.
\end{eqnarray*}
Therefore, $T_*$ is an $\widehat{A}$-isometry. Moreover, taking into account that $\Delta_T \le A$, one has
$$
\Delta_{T_*}=0\oplus (A-T^*AT+\Delta_T)\le 0 \oplus (A+\Delta_T)\le 2 \widehat{A}.
$$ 
Setting $B=2\widehat{A}$, we obtain that $B$ and $T_*$ satisfy the conditions from (iii). Therefore (ii) implies (iii).

We also notice that the lifting $T_*$ of $T$ is minimal, that is $\h_*=\bigvee_{n\ge 0} T_*^n\h$. This fact follows easily (as for $S_{A,T}$ in the proof of Theorem \ref{te21}), keeping in mind that $\D_{\widehat{T}}=\overline{D_{\widehat{T}}A^{1/2}\h}$, $\widehat{T}$ being defined on $\overline{\R(A)}$.

In order to prove that (iii) implies (i) we assume that $T_*$ on $\h_*\supset \h$ and $B$ are as in (iii). Then, using Theorem \ref{te21}, we find a 2-isometric lifting $S=S_{B,T_*}$ as in \eqref{eq28} for $T_*$ on a Hilbert space $\ka \supset \h_*$ such that $\Delta_S=0\oplus B$ on $\ka =\h_*^{\perp}\oplus \h_*$. Clearly, $S$ will be a lifting for $T$ having on $\ka=\h'\oplus \h$ a representation of the form \eqref{eq21} with $W=S|_{\h'}$ and $X\in \mathcal{B}(\h,\h')$. So one obtains $\Delta_S$ on $\ka=\h'\oplus \h$ with representation \eqref{eq22}. Now, since $B\h\subset \h$ and $B\ge 0$, it follows that $\h$ reduces $B$, so $B=B_0\oplus B_1$ on $\h_*=\h^{\perp} \oplus \h$. Then, as $S|_{\h_*^{\perp}}$ is an isometry, we get the representations
$$
\Delta_S=0\oplus B=(0\oplus B_0)\oplus B_1
$$
on the decompositions $\ka=\h_*^{\perp}\oplus \h_*= \h'\oplus \h$, respectively. From the block matrix \eqref{eq22} of $\Delta_S$ we infer that $W^*X=0$. This means that $S$ is a 2-isometric lifting of type II for $T$. Thus we have shown that (iii) implies (i) and all equivalences (i)-(iv) are provided.

Let us remark that the above lifting $S=S_{B,T_*}$ can be chosen in such a manner that it has the form \eqref{eq21} with $\Delta_W=2P$, where $P$ is an orthogonal projection. Indeed, if we take $B=2\widehat{A}$ with $\widehat{A}=I\oplus A$ and $A$ from (ii), as in the implication (ii) $\Rightarrow$ (iii), then $S$ can be expressed in terms of the operator $A$. Keeping in mind the form of $\Delta_{T_*}$, we get
$$
B_{T_*}:= B-\Delta_{T_*}=2 \widehat{A}-\Delta_{T_*}= 2I\oplus (A+T^*AT-\Delta_T)= 2I\oplus B_0
$$
on $\h_*=\ell_+^2(\D_{\widehat{T}})\oplus \h$, where $B_0=A+T^*AT-\Delta_T\ge 0$. So, by \eqref{eq28}, the operator $S$ acts on $\ka =\h_*^{\perp} \oplus \h_*$, where $\h_*^{\perp}=\ell_+^2(\overline{\R(B_{T_*})})$ and $\overline{\R(B_{T_*})}=\ell_+^2(\mathcal{D}_{\widehat{T}})\oplus \overline{\R(B_0)}$. Therefore, $S$ has on $\ka =\h_*^{\perp}\oplus \ell_+^2(\D_{\widehat{T}}) \oplus \h=(\ka \ominus \h)\oplus \h$ the representations
\begin{equation}\label{eq33}
S=
\begin{pmatrix}
S_* & \sqrt{2}\widetilde{I} & \widetilde{B}_0\\
0 & \widehat{S} & DA^{1/2}\\
0 & 0 & T
\end{pmatrix}
=
\begin{pmatrix}
W & X\\
0 & T
\end{pmatrix},
\end{equation}
where $S_*$ is a shift operator on $\h_*^{\perp}$, while $\widetilde{I}$ and $\widetilde{J}$, with $\widetilde{B}_0=\widetilde{J}B_0^{1/2}$, are the canonical injections of $\h_*\ominus \h$, respectively of $\overline{\R(B_0)}$, into $\n(S_*^*)$. We infer from these matrices that $W=S|_{\h'}$ has a special form on $\h'=\ka \ominus \h$ with two shifts on the main diagonal and $\widetilde{I}$ an isometry. Therefore $\Delta_W=0\oplus 2I=2P_{\R(\Delta_W)}$. Also, since $\Delta_S=0\oplus B=0\oplus 2(I\oplus A)$, we get ${\rm cov}(S)=\sqrt{2} \cdot {\rm max}\{1,\|A\|^{1/2}\}$. Thus we have proved a part of the last assertion of the theorem.

It remains to show that the lifting $S$ in (i) can be also chosen minimal, i.e. with $\ka=\bigvee_{n\ge 0} S^n\h$.
Indeed, let us denote by $\ka_0=\bigvee_{n\ge 0} S^n\h$ the smallest invariant subspace in $\ka$ for $S$ which contains $\h$, and let $S_0=S|_{\ka_0}$. Then $S_0$ is a 2-isometry. Since $S^*$ is an extension of $T^*$ and a lifting of $S_0^*$, we have
$$
S_0^*|_{\h}=P_{\ka _0}S^*|_{\h}=T^*.
$$
 Hence $S_0$ is a lifting of $T$. So, $S_0$ has the form
 $$
 S_0=
 \begin{pmatrix}
 W_0 & X_0\\
 0 & T
 \end{pmatrix}
 $$
 on $\ka_0=\h_0\oplus \h$, where $W_0=S_0|_{\h_0}=S|_{\h_0}=W|_{\h_0}$ and $X_0=P_{\h_0}X|_{\h}$. Here $W$ and $X$ are coming from the block matrix \eqref{eq21} of $S$ on $\ka=\h'\oplus \h$ and they have representations of the form
 $$
 W=
 \begin{pmatrix}
 W_0 & \star\\
 0 & \star
 \end{pmatrix},
 \quad X=
 \begin{pmatrix}
 X_0\\
 \star
 \end{pmatrix}
 $$
 on $\h'=\h_0\oplus \h_0^{\perp}$, respectively from $\h$ into $\h_0 \oplus \h_0^{\perp}$ for the matrix of $X$. But $S$ is a lifting of type II for $T$, so $S^*S\h \subset \h$. Then $S_0^*S_0\h=P_{\ka_0} S^*S\h \subset \h$, so $S_0$ is also a lifting of type II for $T$. In addition, $S_0$ is a minimal lifting. Since any lifting of $T$ satisfying (i) also satisfies (iv), it follows that in (iv) one can also chose a minimal lifting for $T$. This completes the proof of theorem.
\end{proof}

\begin{remark}\label{re32}
\rm
The condition $B\h \subset \h$ in the assertion (iii) of Theorem \ref{te31} is essential (as we have seen in the proof of (iii) $\Rightarrow$ (i)). If $T$ has a lifting which has a 2-isometric lifting of type I, then it is not necessarily true that $T$ has a 2-isometric lifting of type II.
\end{remark}

Obviously, Theorem \ref{te31} generalizes Theorem \ref{te21}. The assertion (ii) of the latter can be reformulated in terms of $A$-contractions and $Q$-expansive operators $T$, i.e. operators satisfying $0\le Q \le T^*QT$. More precisely, from Theorem \ref{te21} and \cite[Theorem 2.1]{BFF}, we deduce the following result which shows when an operator with 2-isometric liftings of type II has a lifting of type I. This always happens for expansive operators (by Theorems \ref{te25} and \ref{te31}). In addition, the corollary below reproves the assertion (ii) of Theorem \ref{te25} in the case when $Q$ is a scalar multiple of the identity (for $T$ expansive).

\begin{corollary}\label{co33}
An operator $T\neq 0$ has a 2-isometric lifting of type I if and only if there exist two positive operators $A$ and $Q$ satisfying $\Delta_T\le A$ and $T^*AT\le A\le Q \le T^*QT$.
\end{corollary}

Notice that if $T$ is an operator similar to a contraction, then $T$ is an $A_0$-contraction for an invertible operator $A_0$, which can be chosen such that $A_0\ge T^*T$. As a direct consequence of Theorem \ref{te31}, we obtain the following result.
\begin{corollary}\label{co35}
If $T\in \B$ is non-contractive and similar to a contraction, then $T$ has a 2-isometric lifting of type II and of covariance $\sqrt{2} \cdot \|A_0\|^{1/2}$, with $\|A_0\|\ge \|T\|$.
\end{corollary}
Among the operators similar to contractions we can consider those having ${\rho}$-dilations. For $\rho >0$, an operator $T\in \B$ is said to have a ${\rho}$-\emph{dilation} if there exists a unitary operator $U_{\rho}$ on some space $\h_{\rho}\supset \h$ such that $T^n=\rho P_{\h}U^n_{\rho}|_{\h}$ for $n\ge 1$ (see \cite{SFb}). Such a $\rho$-dilation $U_{\rho}$ of $T$ is not a lifting for $T$. From Corollary \ref{co35} we deduce the following
\begin{corollary}\label{co36}
Any operator $T$ having a ${\rho}$-dilation has a minimal 2-isometric lifting of type II.
\end{corollary}
Recall (see \cite{SuSu}) that an operator with a $\rho$-dilation is a $S_T$-isometry for some positive operator $S_T$. So if $S_T \ge \Delta_T$ it follows that $T$ has even a 2-isometric lifting of type I.

Another special class of operators similar to contractions is given by quasi-contractions, that is the $T^*T$-contractions (see \cite{BS2, CS}). For such an operator $T$ one obtains from Theorem~\ref{te31} a $2$-isometric lifting $S$ of type II with 
$${\rm cov}(S)=\sqrt{2} \cdot {\rm max} \{1,\|T\|\}=\sqrt{2} \|T\| ,$$ whenever $T$ is non-contractive.

Now we show that there exist operators similar to contractions, even quasi-contractions, which do not have 2-isometric liftings of type I.
\begin{example}\label{ex37}
\rm
Let $C$ be a contraction on $\h$ such that $C^n \to 0$ strongly and let $\delta > 1$ be a scalar. Then the operator $T$ on $\h \oplus \h$ with the matrix representation
$$
T=
\begin{pmatrix}
C & \delta J\\
0 & 0
\end{pmatrix},
$$
where $J(0\oplus h)=h \oplus 0$, for $h\in \h$, is a non-contractive quasi-contraction. Thus $T$ is similar to a contraction. But $T^n\to 0$ strongly, because $C$ has this property. Hence $T$ cannot be an $A$-isometry with $A\neq 0$. Thus, by Theorem \ref{te21}, $T$ does not possess 2-isometric liftings of type I. A simpler example of this form is a nilpotent operator of order 2 (and thus having $\rho$-dilations for suitable $\rho$) on $\C \oplus \C$.
\end{example}

On the other hand, we show that an operator $T$ similar to a contraction can have 2-isometric liftings of type I (not only of type II), without being similar to an isometry. Therefore, in this case, Corollary \ref{co33} applies to a non-invertible operator $A$, so not equal to the operator $A_0$ from Corollary \ref{co35} (otherwise $T$ will be similar to an isometry by  \cite[Theorem 2.1]{BFF}).
\begin{example}\label{ex38}
\rm
Let $U$ be a unitary operator on $\h=\h_0\oplus \h_1$ ($\h_j\neq \{0\}$) with a block matrix of the form
$$
U=
\begin{pmatrix}
V & \star\\
0 & V'^*
\end{pmatrix},
$$
where $V$ and $V'$ are isometries. Let $C=V\oplus 0$ on $\h$ and set $E=-J(0\oplus V'^*)$, acting from $\{0\} \oplus \h$ to $\h\oplus \{0\}$, where $J$ is as in Example \ref{ex37}. Finally, we define the operator $T$ on $\h\oplus \h$ by
$$
T=
\begin{pmatrix}
C & E\\
0 & U
\end{pmatrix}.
$$
It is clear that $C^*E=0$. Therefore $T$ has the form \eqref{eq24}. Furthermore, if $Z=JP_{\h_1}$ (an operator from $\{0\}\oplus \h$ into $\h\oplus \{0\}$), then it is easy to see that
$$
CZ-ZU=-J(0\oplus V'^*)=E.
$$
A known result (see  \cite{B2}) and the last relation ensure that $T$ is similar to a contraction, more precisely to the diagonal operator $C\oplus U$. As $C$ is not an isometry, the operator $T$ is not similar to an isometry. However, $T$ has a minimal Brownian isometric lifting of type I by Corollary \ref{co22}.
\end{example}
\subsection{Isomorphic minimal 2-isometric liftings.}
Previous results show the existence of minimal 2-isometric liftings, but their uniqueness up to an isomorphism (i.e. a unitary equivalence which fixes $\h$) is not guaranteed, in this context. This is in contrast to the classical case of isometric (unitary) dilation theory of contractions (see \cite{FF}, \cite{SFb}). However, the following fact about minimal 2-isometric liftings of type II is true.
\begin{proposition}\label{pr39}
Let $T\in \B$ and let $S$ on $\ka\supset \h$ and $S'$ on $\ka'\supset \h$ be two minimal 2-isometric liftings of $T$ with $S^*S\h \subset \h$ and $S'^*S'\h \subset \h$. Then $S$ and $S'$ are isomorphic if and only if $S^*S|_{\h}=S'^*S'|_{\h}$. If this is the case, then the 2-isometries $S|_{\ka \ominus \h}$ and $S'|_{\ka' \ominus \h}$ are unitarily equivalent.
\end{proposition}
\begin{proof}
Let $T,S$ and $S'$ be as above. Since $S$ is a 2-isometry, we have $S^{*2}S^2=2S^*S-I$ and, for $n>2$, one obtains the formula $S^{*n}S^n=nS^*S-(n-1)I$. A similar relation holds for $S'$. Since $S^*S\h \subset \h$, we infer that $S^{*n}S^n\h \subset \h$ for $n\ge 1$. On the other hand, we have $P_{\h}S^n|_{\h}=T^n=P_{\h}S'^n|_{\h}$ for $n\ge 1$. Using these relations, we obtain, for any finite system $\{h_j\}_0^n\subset \h$, that
$$
\|\sum_{j=0}^nS^jh_j\|^2_{\ka} = \sum_{j,l=0}^n \langle S^jh_j, S^lh_l\rangle = \|h_0\|^2 + 
$$
$$
\sum_{l=1}^n \langle h_0, T^lh_l\rangle + \sum_{l\ge j\ge 1} \langle S^{*(l-j)}S^{*j}S^jh_j,h_l \rangle + \sum_{j=1}^n \langle T^jh_j,h_0 \rangle +  \sum_{j>l\ge 1}\langle h_j,S^{*(j-l)}S^{*l}S^lh_l \rangle
$$
$$
=\|h_0\|^2 +2 {\rm Re} \sum_{j=1}^n \langle T^j h_j, h_0\rangle + \sum_{l\ge j \ge 1} \langle S^{*j}S^jh_j, T^{l-j}h_l \rangle + \sum_{j>l\ge 1} \langle T^{j-l}h_j,S^{*l}S^lh_l \rangle
$$
$$
= \|h_0\|^2 +2 {\rm Re} \sum_{j=1}^n \langle T^j h_j, h_0\rangle - \sum_{l\ge j \ge 2} (j-1) \langle h_j,T^{l-j}h_l \rangle - \sum_{j>l \ge 2} (l-1) \langle T^{j-l}h_j, h_l \rangle
$$
$$
 + \sum_{l\ge j \ge 1} j\langle S^*Sh_j, T^{l-j}h_l \rangle + \sum_{j>l\ge 1} l \langle T^{j-l}h_j, S^*Sh_l \rangle .
$$

Assume now that $S^*S|_{\h}=S'^*S'|_{\h}$. Then the last expression in the above computation can be also written in terms of $T$ and $S'$. So, proceeding in a reverse way, one obtains
$$
\|\sum_{j=0}^n S^jh_j\|_{\ka}= \|\sum_{j=0}^n S'^jh_j\|_{\ka'}.
$$
Then, by a standard argument, it follows that there exists a unitary operator $Z$ from $\ka$ onto $\ka'$ satisfying the relations $ZS=S'Z$ and $Z|_{\h}=I$. Therefore $S$ and $S'$ are isomorphic as 2-isometric liftings of $T$.

Conversely, if there is such an operator $Z$ which preserves the elements of $\h$ and intertwines $S$ with $S'$, then $S^*S=Z^*S'^*S'Z$. We get $S^*S|_{\h}=Z^*S'^*S'|_{\h}=S'^*S'|_{\h}$, because $S'^*S'\h \subset \h$ and $Z|_{\h}=I=Z^*|_{\h}$. Thus the first assertion of Proposition~\ref{pr39} is proved. Furthermore, if $Z$ is as above, then $Z(\ka \ominus \h)=\ka'\ominus \h$ and $Z':=Z|_{\ka \ominus \h}$ is unitary. So one obtains
$$
Z'(S|_{\ka \ominus \h})=(ZS)|_{\ka \ominus \h}=(S'Z)|_{\ka \ominus \h}=(S'|_{\ka '\ominus \h})Z'.
$$
Therefore the 2-isometries $S|_{\ka \ominus \h}$ and $S'|_{\ka ' \ominus \h}$ are unitarily equivalent by $Z'$.
\end{proof}

From the last assertion of the previous proposition we derive the following

\begin{corollary}\label{co310}
Suppose that two minimal 2-isometric liftings of type II of an operator are isomorphic. Then one of them is of type I if and only if the other is of type I.
\end{corollary}

\begin{remark}\label{re311}
\rm
Assume that $S$ on $\ka \supset \h$ is a 2-isometric lifting of type I for $T\in \B$. If $S$ does not satisfy the minimality condition, then the minimal lifting $S_0=S|_{\ka_0}$ on $\ka_0=\bigvee_{n\ge 0} S^n\h$ is also of type I for $T$. Indeed, since $\ka_0$ and $\ka \ominus \h$ are invariant for $S$, it follows that $S_0(\ka _0 \ominus \h)=S(\ka _0 \ominus \h) \subset \ka _0 \ominus \h$, and so $S_0|_{\ka_0 \ominus \h}=S|_{\ka_0 \ominus \h}$ is an isometry.

Moreover, $\ka_0$ is a reducing subspace for $S$. Indeed, since $S^*S \h \subset \h$, we have
$$
S^*\ka_0 \subset \h \vee (\bigvee_{n\ge 2} S^*S^n\h)\subset \ka_0 \vee (\bigvee_{n\ge 1} \Delta_S S^n\h) \subset \ka_0.
$$
The last inclusion holds because $\ka \ominus \h \subset \n(\Delta_S)$, $S$ being of type I and so by \eqref{eq21} we have $\Delta_S=0\oplus \Delta_{S}|_{\h}$ on $\ka=(\ka \ominus \h)\oplus \h$, hence $\Delta_SS^nh= \Delta_ST^nh \in \h$ for $h\in \h$, $n\ge 1$.
Thus the minimality condition for 2-isometric liftings of type I can be defined with respect to reducing subspaces, or equivalently, to subspaces assumed to be only invariant. This fact is analogous to the notion of minimal Brownian unitary (respectively isometry, in the cyclic case) extension for a 2-isometry, which appears in  \cite[Section 10]{AS3}.
\end{remark}
\begin{acknowledgements} 
The first named author was supported in part by
 the project FRONT of the French
National Research Agency (grant ANR-17-CE40-0021) and by the Labex CEMPI (ANR-11-LABX-0007-01).
The second named author was supported by a project financed by Lucian Blaga
University of Sibiu and Hasso Plattner Foundation research Grants LBUS-IRG-2019-05. This paper has been finalized when both authors were in residence at CIRM, Luminy, France 
during a ``Recherche en Bin\^ome no. 2198'' programme.
\end{acknowledgements}


\begin{thebibliography}{99}

\bibitem{Ag} \textsc{J. Agler}, {\it An abstract approach to model theory}, in Surveys of Some Recent Results
in Operator Theory, vol. II, Pitman Res. Notes Math. Ser., vol. 192, John Wiley and Sons, New York 1988, pp. 1--23.

\bibitem{AS1} \textsc{J. Agler, M. Stankus}, {\it $m$-isometric transformations of Hilbert spaces}, Integral Equations Operator Theory {\bf 21}, 4 (1995), 383--429.

\bibitem{AS2} \textsc{J. Agler, M. Stankus}, {\it $m$-isometric transformations of Hilbert spaces II}, Integral Equations Operator Theory {\bf 23}, 10 (1995), 1--48.

\bibitem{AS3} \textsc{J. Agler, M. Stankus}, {\it $m$-isometric transformations of Hilbert spaces III}, Integral Equations Operator Theory {\bf 24} (1996), 379--421.

\bibitem{Aleman} \textsc{A. Aleman}, {\it The multiplication operator on Hilbert spaces of analytic functions},
Habilitationsschrift, Fern Universit\"{a}t, Hagen, 1993.

\bibitem{AR} \textsc{A. Aleman, W. Ross}, {\it The backward Shift on weighted Bergman spaces}, Michigan Math. J. {\bf 43}, (1996), 291--319.

\bibitem{ACJS} \textsc{A. Anand, S. Chavan, Z.J. Jablonski, J. Stochel}, {\it A solution to the Cauchy dual subnormality problem for 2-isometries}, J. Func. Anal. {\bf 277}(2019), Issue 12, 108292.

 \bibitem{B1} \textsc{C. Badea}, {\it Operators near completely polynomially dominated ones and similarity problems}, J. Operator Th. {\bf 49} (2003), 3--23.

\bibitem{B2} \textsc{C. Badea}, {\it Perturbations of operators similar to contractions and the commutator equation},  Studia Math. {\bf 150} (2002), 273--293.

\bibitem{BS1} \textsc{C. Badea, L. Suciu}, {\it Similarity problems, F\o lner sets and isometric representations of amenable semigroups}, Mediterr. J. Math. {\bf 16}(2019), article no. 5.

\bibitem{BS2} \textsc{C. Badea, L. Suciu}, {\it The Cauchy dual and $2$-isometric liftings of concave operators} J. Math. Anal. Appl. {\bf 472} (2019), 1458--1474.

\bibitem{BP} \textsc{H. Bercovici, B. Prunaru}, {\it Quasiaffine transforms of polynomially bounded operators}, Arch. Math. (Basel) {\bf 71} (1998) 384--387.

\bibitem{BFF} \textsc{A. Biswas, A. E. Frazho, C. Foias}, {\it Weighted commutant lifting}, Acta Sci.  Math. (Szeged), {\bf 65},
3–4 (1999), 657--686.

\bibitem{CS} \textsc{G. Cassier, L. Suciu}, {\it Mapping theorems and similarity to contractions for classes of $A$-contractions},
in: Hot Topics in Operator Theory, Theta Ser. Adv. Math. (2008), 39--58.

\bibitem{D} \textsc{R.G. Douglas}, {\it On majorization, factorization and range inclusion
of operators in Hilbert space}, Proc. Amer. Math. Soc. {\bf 17} (1966) 413--416.

\bibitem{Do} \textsc{R.G. Douglas}, {\it On the operator equation $S^*XT=X$ and related
   topics}. Acta. Sci. Math. (Szeged) {\bf 30} (1969), 19--32.

\bibitem{FF} \textsc{C Foias, A. E. Frazho}, {\it The Commutant Lifting Approach to Interpolation Problems}, Birkh\"auser Verlag, Basel-Boston-Berlin, 1990.

\bibitem{FFK2} \textsc{C. Foias, A. E. Frazho, M. A. Kaashoek}, {\it Contractive liftings and the commutator}, C. R. Acad. Sci. Paris, Ser. I {\bf 335} (2002) 431--436.

\bibitem{GO} \textsc{O. Giselsson, A. Olofsson}, {\it On some Bergman shift operators}, Complex Anal. Oper. Theory {\bf 6} (2012) 829--842.

\bibitem{JKKL} \textsc{S. Jung, Y. Kim, E. Ko, J. E. Lee}, {\it On $(A,m)$-expansive operators}, Studia Math. {\bf 213}, 1, (2012), 3--23.

\bibitem{Ke} \textsc{L. Kerchy}, {\it Generalized Toeplitz operators}, Acta  Sci. Math. (Szeged) {\bf 68} (2002), 373--400.

\bibitem{Kub-1997} \textsc{C.S. Kubrusly}, {\it An Introduction to Models and Decompositions in Operator Theory}, Birkh\"{a}user,
Boston, 1997.

\bibitem{MMS} \textsc{W. Majdak, M. Mbekhta, L. Suciu}, {\it Operators intertwining with isometries
and Brownian parts of 2-isometries}, Linear Algebra Appl. {\bf 509} (2016), 168--190.

\bibitem{MajSu} \textsc{W. Majdak, L. Suciu}, {\it Brownian isometric parts of concave operators}, New York J. Math. {\bf 25} (2019) 1067-1090.

\bibitem{MS} \textsc{M. Mbekhta, L. Suciu}, {\it Classes of operators similar to partial isometries}, Integr. Equat. Oper. Th. vol. {\bf 63}, Number 4, (2009), 571--590.

\bibitem{McC} \textsc{S. McCullough}, {\it SubBrownian operators}, J. Operator Th. {\bf 22} (1989), 291--305.

\bibitem{McCR} \textsc{S. McCullough, B. Russo}, {\it The 3-Isometric Lifting Theorem}, Integr. Equat. Oper. Th. {\bf 84}, 1 (2016), 69--87.

\bibitem{M} \textsc{V. M\"{u}ller}, {\it Models for operators using weighted shifts}, J. Operator Th. {\bf 20}, 1 (1988), 3--20.

\bibitem{Ol} \textsc{A. Olofsson}, {\it A von Neumann-Wold decomposition of two-isometries}, Acta Sci. Math. (Szeged) {\bf 70} (2004), 715--726.

\bibitem{Richter} \textsc{S. Richter}, {\it Invariant subspaces of the Dirichlet shift}, J. Reine Angew. Math. {\bf 386} (1988), 205--220.

\bibitem{Richter0} \textsc{S. Richter},  {\it A representation theorem for cyclic analytic two-isometries}, Trans.
Amer. Math. Soc. {\bf 328} (1991), 325--349.

\bibitem{Rydhe} \textsc{E. Rydhe}, {\it Cyclic m-isometries, and Dirichlet type spaces}, 
J. London Math. Soc. (2) 99 (2019) 733--756.

\bibitem{Sh} \textsc{S. Shimorin}, {\it Wold-type decompositions and wandering subspaces for operators close to isometries}, J. Reine Angew. Math. {\bf 531} (2001), 147--189.

\bibitem{SS} \textsc{O. A. M. Sid Ahmed, A. Saddi}, {\it $A-m$-Isometric operators in semi-Hilbertian spaces}, Linear Algebra Appl. {\bf 436} (10) (2012), 3930--3942.

\bibitem{Stankus} \textsc{M. Stankus}, {\it $m$-Isometries, $n$-Symmetries and other linear transformations which are hereditary roots},  Integr. Equ. Oper. Th. {\bf 75} (2013), 301--321.

\bibitem{S-2009} \textsc{L. Suciu}, {\it Maximum $A$-isometric part of an $A$-contraction and applications}, Israel J. Math. {\bf 174}
(2009), 419--442.

\bibitem{S-2020} \textsc{L. Suciu}, {\it On Operators with Two-Isometric Liftings}, Complex Analysis and Operator Theory (2020) 14:5, 1--16.

\bibitem{SuSu} \textsc{L. Suciu, N. Suciu}, {\it Asymptotic behaviours and generalized Toeplitz operators}, J. Math. Anal. Appl., vol. {\bf 349}, issue 1, (2009), 280--290.

\bibitem{SFb} \textsc{B. Sz.-Nagy, C. Foias, H. Bercovici, L. K\'erchy}, {\it Harmonic Analysis
of Operators on Hilbert Space}, Revised and enlarged edition, Universitext, Springer, New York, 2010.

\bibitem{TV} \textsc{S. Treil, A. Volberg}, {\it A fixed point approach to Nehari's problem and its
applications}, Operator Theory: Advances and Applications {\bf 71}, Birkh\"auser,
1994, 165--186.

\end{thebibliography}
\end{document}